\documentclass[a4paper]{article}

\usepackage{
amsmath,
amsthm,
amscd,
amssymb,
}
\usepackage{stmaryrd}
\usepackage{comment}
\usepackage{tikz}
\usetikzlibrary{positioning,arrows,calc}
\usepackage{xspace}
\usepackage{wasysym}

\usepackage{graphicx}

\usepackage{url}

\theoremstyle{plain}
\newtheorem{lemma}{Lemma}
\newtheorem{definition}{Definition}

\newtheorem{proposition}{Proposition}
\newtheorem{theorem}{Theorem}

\newtheorem{question}{Question}
\newtheorem{remark}{Remark}

\newtheoremstyle{derp}
{3pt}
{3pt}
{}
{}
{\upshape}
{:}
{.5em}
{}
\theoremstyle{derp}
\newtheorem{example}{Example}

\newcommand{\Q}{\mathbb{Q}}
\newcommand{\Z}{\mathbb{Z}}

\newcommand{\N}{\mathbb{N}}

\newcommand{\id}{\mathrm{id}}

\newcommand\xqed[1]{%
  \leavevmode\unskip\penalty9999 \hbox{}\nobreak\hfill
  \quad\hbox{#1}}
\newcommand\qee{\xqed{$\fullmoon$}}

\newcommand{\Aut}{\mathrm{Aut}}

\newcommand{\Homeo}{\mathrm{Homeo}}

\newcommand{\gates}{\mathfrak{G}}
\newcommand{\egates}{\hat{\mathfrak{G}}}
\newcommand{\gnets}{\mathfrak{L}}
\newcommand{\egnets}{\hat{\mathfrak{L}}}

\newcommand{\Sym}{\mathrm{Sym}}
\newcommand{\Alt}{\mathrm{Alt}}

\newcommand{\follow}{\mathcal{F}}

\newcommand{\SIAut}{\mathrm{SIAut}}
\newcommand{\SAut}{\mathrm{SAut}}

\newcommand{\diff}{\mathrm{diff}}

\title{Gate lattices and the stabilized automorphism group}

\author{
Ville Salo \\
vosalo@utu.fi
}

\begin{document}
\maketitle

\begin{abstract}
We study the stabilized automorphism group of a subshift of finite type with a certain gluing property called the eventual filling property, on a residually finite group $G$. We show that the stabilized automorphism group is simply monolithic, i.e.\ it has a unique minimal non-trivial normal subgroup -- the monolith -- which is additionally simple. To describe the monolith, we introduce gate lattices, which apply (reversible logical) gates on finite-index subgroups of $G$. The monolith is then precisely the commutator subgroup of the group generated by gate lattices. If the subshift and the group $G$ have some additional properties, then the gate lattices generate a perfect group, thus they generate the monolith. In particular, this is always the case when the acting group is the integers. In this case we can also show that gate lattices generate the inert part of the stabilized automorphism group. Thus we obtain that the stabilized inert automorphism group of a one-dimensional mixing subshift of finite type is simple. 
\end{abstract}

\section{Introduction}


In Section~\ref{sec:SymDynIntro}, we begin with the symbolic dynamical background for our group, by recalling the dimension group representation and the stabilized automorphism group. In Section~\ref{sec:GroupIntro}, we concentrate on the more intrinsic features of our proof. 

\subsection{The stabilized automorphism group of a mixing SFT}
\label{sec:SymDynIntro}

Some of the primary objects of study in the field of symbolic dynamics are the (one-dimensional) \emph{subshifts of finite type}, or \emph{SFTs}. These are (up to appropriate isomorphism) the topological dynamical systems whose points are bi-infinite paths $p : \Z \to E$ in a finite directed graph $(V, E)$, where the constraint on paths is that the edge $p(n)$ must end at the vertex where $p(n+1)$ begins for all $n$, and the dynamics is given by the shift map which simply reparametrizes the path by the formula $\sigma(p)_i = p_{i+1}$. We always assume the graph is \emph{essential}, meaning the in- and out-degree of every node is positive (i.e.\ there are no sources or sinks).

Of special interest are the \emph{mixing} SFTs, which are characterized by the transition matrix of the graph being primitive (meaning some power of it has only positive entries). An important special case, known as a \emph{full shift}, is obtained when the graph has a single vertex. We may write this as simply $(A^\Z, \sigma)$ (taking loops $E = A$ at the vertex as edges).


The \emph{automorphism group} $\Aut(X, \sigma) = \Aut(X)$ of a mixing SFT $X$ is the group of homeomorphisms $f : X \to X$ which commute with the shift map. This is a classical object in the field of symbolic dynamics, introduced by Hedlund \cite{He69} for the full shift, and studied for general mixing SFTs in \cite{BoLiRu88,KiRo90}. The structure of this group is not well understoood, for example, it is an open problem whether the automorphism groups of $\{0, 1\}^\Z$ and $\{0, 1, 2\}^\Z$ are isomorphic.

As is typical with complicated large groups, it is of interest to try to understand the quotients of the group, or equivalently the normal subgroups. The most important known quotient is perhaps the dimension group representation. To define this, recall that one of the main invariants of isomorphism of mixing SFTs is the \emph{dimension group} 
introduced by Krieger \cite{Kr80}.

For a careful treatment of the dimension group see \cite[§7.5]{LiMa95}, we only give a quick overview here. For our purposes a sufficient definition of this group is 
$\Delta_M = \underset{\raisebox{0.02cm}{\footnotesize$\rightarrow$}}{\lim}{}_n \Z^d$,
where we use the defining matrix $M$ of $X$ for the transitions between the groups $\Z^d$. 
The matrix $M$ gives a group automorphism $\delta_M$ of $\Delta_M$, and the resulting pair $(\Delta_M, \delta_M)$ is an isomorphism invariant for $X$ (in particular, it does not depend on the choice of $M$). Now, $\Aut(X)$ induces a natural action on $\Delta_M$ by automorphisms, with $\sigma$ mapping to $\delta_M$, and the resulting representation is independent of the choice of $M$. Furthermore, for each automorphism of $X$, the automorphism induced on $\Delta_M$ is ``positive'', which is a certain index-$2$ restriction \cite{Bo08}. The representation of $\Aut(X)$ by positive $\delta_M$-commuting automorphisms of $\Delta_M$ is the \emph{dimension group representation}. 

Automorphisms in this kernel of this representation called \emph{inert}. It is generally agreed that the kernel, i.e.\ the group of inert elements, contains most of the ``complexity'' of the automorphism group. In this and the following paragraph, we give some concrete arguments in this direction. First, the dimension group representation is always far from faithful: for example, the group $\Aut(A^\Z)$ for any finite alphabet $A$ can be embedded in the kernel of the dimension group representation of $\Aut(X)$, whenever $X$ is a mixing SFT (or even an uncountable sofic shift), as the standard embeddings \cite{KiRo90,Sa18d} take place inside the group of inert automorphisms. See also Section~3.6. Interestingly, while we know the representation is far from injective, it is in many cases surjective \cite{BoLiRu88,Lo13} (though not always, by the important example of \cite{KiRo92}).

Second, the image of the dimension representation is a simpler ``kind'' of group than the kernel. For example for $\{0,1\}^\Z$, the image is isomorphic to $\Z$, and for any full shift it is isomorphic to $\Z^d$ for some $d \geq 1$. In general, we can at the very least deduce from the definition of the dimension group that the image of the dimension representation is a \emph{linear group}, i.e.\ it acts faithfully by automorphisms on a finite-dimensional vector space over a field (to see this, we observe that the direct limit embeds into $\Q^d$, and its automorphisms correspond to linear maps). On the other hand, it is known that already finitely-generated subgroups of the automorphism group (and thus also the kernel of the dimension group representation) exhibit many behaviors that are impossible in linear groups \cite{BoLiRu88,Sa19a,Sa22b}. 


We list some other problems about the automorphism group (see also e.g.\ \cite{Bo08,BoSc21} and \cite[§13.2]{LiMa95}). The inert part of $\Aut(X)$ is often thought of as eliminating some kind of average information flow, and in some special cases (like full shifts), this is true in a quite literal sense \cite{Ka96,Sc20}. On the other hand, the elements of finite order in $\Aut(X)$ also intuitively have no average information flow. The so-called \emph{finite-order generation (FOG) conjecture} predicted that the inert subgroup is generated by its finite order elements. It was resolved in the negative for some mixing SFTs in \cite{KiRoWa97} (see also \cite{KiRo91} for the resolution of a predecessor of this conjecture), but is still open on many individual subshifts, for example for $\{0,1\}^\Z$. 
The \emph{virtual FOG conjecture} of whether the finite order elements generate a finite index subgroup in the inert subgroup, is still open in general.

Other interesting questions are what the relationship is between the inert subgroup and simple automorphisms of Nasu \cite{Na88}, and what the relation is between the commutator subgroup $[\Aut(X), \Aut(X)]$ and the groups discussed above.

All these mysteries are greatly simplified (and mostly solved) in the point of view of the \emph{stabilized automorphism group} introduced by Hartman, Kra and Schmieding in \cite{HaKrSc21}. The stabilized automorphism group of a subshift $X$ is simply the union of the groups $\Aut(X, \sigma^n)$ for all $n \geq 1$, and we will denote it by $\SAut(X)$. Like $\Aut(X)$, this is a countably infinite group, is not finitely generated, and (up to isomorphism) is invariant under topological conjugacy. 

There is an obvious analogue of the dimension group representation for the stabilized automorphism group, which maps $\SAut(X)$ into the union of positive automorphism groups of dimension groups of the $(X, \sigma^n)$. The kernel of this map is a natural analog of inert automorphisms in the stabilized context, and we again call the elements of this kernel inert. In the stabilized setting, the dimension group representation is always surjective \cite{BoLiRu88}.

The analog of the FOG conjecture is true in the stabilized setting -- the inert part, which we will denote by $\SIAut(X)$, is generated by its finite-order elements \cite{Wa90a}. Simple inert automorphisms \cite{Na88} give a generating set as well (see Section~\ref{sec:Inertness}). Furthermore, the abelianization factors through the dimension group representation, i.e.\ inert elements are in the commutator subgroup. More simplifications provided by the stabilized point of view are listed in \cite{HaKrSc21}.

One thing that is not resolved as prettily as one might like is that the dimension group representation is not quite the abelianization of $\SAut(X)$, i.e.\ $[\SAut(X), \SAut(X)] \neq \SIAut(X)$ in general, since irreducible SFTs can have non-abelian dimension representations \cite{BoLiRu88}.

These technical observations are older than the explicit study of the stabilized automorphism group (see Section~\ref{sec:Inertness}), but the introduction of the stabilized automorphism group as a first-class object has lead to very interesting new discoveries. In particular, Schmieding has shown in \cite{Sc22} that the stabilized automorphism group of a mixing SFT in a sense remembers its entropy, allowing us to prove that for example full shifts with alphabets $\{0,1\}$ and $\{0,1,2\}$ have distinct stabilized automorphism groups (which stays open for the usual automorphism groups).

The above discussion suggests that the stabilized automorphism group is important, and its inert part is somehow a very robust object, which can be defined in many ways. The following is one of the main results (and certainly the most difficult result) of \cite{HaKrSc21}, and was the starting point for our work. 

\begin{theorem}[\cite{HaKrSc21}]
\label{thm:FullShiftSimple}
For any one-dimensional full shift $X = A^\Z$, the group $\SIAut(X)$ is simple.
\end{theorem}

Our result generalizes this to all mixing SFTs, and also clarifies the robustness of the inert elements of the stabilized automorphism group.

\begin{theorem}
\label{thm:OneDimensionalSimple}
For any mixing one-dimensional SFT, the group $\SIAut(X)$ is simple. It is the unique minimal normal subgroup of $\SAut(X)$.
\end{theorem}

This shows that $\SIAut(X)$, being the unique minimal normal subgroup, is a special (in particular characteristic) subgroup inside $\SAut(X)$, and can be picked out without knowing its dynamical nature. As mentioned above, it is somewhat disappointing that $\SIAut(X)$ is {not} in general the commutator subgroup of $\SAut(X)$, as one might expect from how large infinite groups usually behave. It turns out, however, that it is the commutator subgroup of a different large group -- the group of gate lattices.

\subsection{Stabilized automorphism group through gate lattices}
\label{sec:GroupIntro}

Let $G$ be a countable group, let $\Sigma$ be a finite (discrete) alphabet, and let $\Sigma^G$ carry the product topology. We consider a subshift $X \subset \Sigma^G$, i.e.\ a set which is a topologically closed and shift-invariant, meaning $G$-translations $(g \cdot x)_h = x_{hg}$ map points of $X$ to points of $X$. The \emph{gates} on $X$ are the homeomorphisms $\chi : X \to X$ which only modify a bounded set of coordinates (for all $x \in X$, we have $\chi(x)_g = x_g$ for all $g \notin F$, for some finite set $F$).


Gates were introduced in \cite{Sa22a} in the one-dimensional case, but variants of them, and more generally ``movement inside the asymptotic relation'', appears in some form in many parts of the symbolic dynamics literature, see for example \cite{PeSc97,Kr83,LiSc99,Sc97a,ChMe16,Kr80a}. We show in the appendix that gates can be interpreted as the topological full group of the asymptotic relation, i.e.\ they are a special case of a very classical idea.

As a new concept, we introduce gate lattices: If $\chi$ is a gate, we can apply it ``at'' $g \in G$ by conjugating $\chi$ with the $g$-translation. If $H$ is a sufficiently sparse finite-index subgroup, these applications commute, and we obtain a well-defined homeomorphism $\chi^H$ by ``applying each $\chi^g$ simultaneously''; such a map is called a \emph{gate lattice}. Here, one can read the word ``lattice'' as a synonym of a finite-index subgroup, so a gate lattice is a gate applied on a lattice.\footnote{Actually, we allow the gate to be applied on a right translate of a finite-index subgroup, but these two definitions give the same group of gate lattices by Lemma~\ref{lem:NoTranslate}.} We write $\gnets(X)$ for the group generated by gate lattices. Of even more importance is the subgroup of \emph{even gate lattices} $\egnets(X)$, meaning roughly that the gates $\chi$ used should perform even permutations ``in all contexts''.

Although our algebraic framework for gate lattices seems to be new, gate lattices do also appear in the literature in various forms. In particular a lot of work on inert automorphisms can be seen through this lens, as we will see in Section~\ref{sec:Inertness} (and which to some extent is precisely the point of this paper). The block representations considered in \cite{Ka96,Ka99} essentially study gate lattices on $\Z^d$ (though the definition is slightly more restrictive). Variants of gate lattices also appear in \cite{BrGaTh20,ArMoEo22}.


Generalizing automorphism groups of $\Z$-subshifts, if $G$ is a countable group, the group $\Aut(X)$ of $G$-commuting homeomorphisms on a subshift $X \subset \Sigma^G$ is a countably infinite group, which typically has quite complicated structure. If $G$ is further residually finite (throughout the paper, we assume $G$ is both residually finite and countably infinite, unless otherwise specified), there is a natural way to define the stabilized automorphism group: for finite-index subgroups $H < G$, the systems $(X, H)$ are $H$-subshifts (up to topological conjugacy), and $\Aut(X, H)$ are groups of homeomorphisms on $X$, so it is natural to define $\SAut(X) = \bigcup_{H, [G : H] < \infty} \Aut(X, H)$. Unlike on $\Z$, there is no standard notion of inertness for the automorphism group even if $X = \Sigma^G$ is a full shift, so it is not easy to define $\SIAut(X)$.

We can relate gate lattices to stabilized automorphism groups in a straightforward way. If $H$ a subgroup of $G$, it turns out that the gate lattice $\chi^H$ (when it is well-defined) commutes with the action of $H$, thus $\chi^H \in \SAut(X)$. It follows that gate lattices form a subgroup of $\SAut(X)$.

We concentrate on the case where $X$ is a subshift of finite type (defined by a finite set of forbidden patterns), which has the ``eventual filling property'' or EFP (Definition~\ref{def:EFP}). In one dimension, i.e.\ when $G = \Z$, mixing and EFP are equivalent for SFTs.

Our main result is the following. 

\begin{theorem}
\label{thm:CommutatorIsSimple}
For any residually finite countably infinite group $G$ and any EFP SFT $X \subset \Sigma^G$,
\begin{itemize}
\item $\egnets(X)$ is simple,
\item $\egnets(X)$ equals the commutator subgroup of $\gnets(X)$,
\item $\egnets(X)$ and $\gnets(X)$ are normal in $\SAut(X)$, and
\item $\egnets(X)$ is the monolith of $\gnets(X)$ and $\SAut(X)$.
\end{itemize}
\end{theorem}

This is proved in Theorem~\ref{thm:MainProof}. Here, recall that a group is called \emph{monolithic} if it has a unique minimal non-trivial normal subgroup, that subgroup being called the \emph{monolith}. In this terminology, the theorem says that $\egnets(X)$ is the monolith of $\gnets(X)$ and $\SAut(X)$, and these latter groups are monolithic. In the terminology of \cite{Ko69}, they are \emph{simply monolithic}, i.e.\ they are monolithic and their monolith is simple. In universal algebraic terms, being monolithic group is the same as \emph{subdirectly irreducible} \cite{BuSa81}, meaning that the group cannot be factored nontrivially as a subdirect product. 

Our result shows in particular that $\egnets(X)$ is a characteristic subgroup in $\SAut(X)$, thus every isomorphism between stabilized automorphism groups $\SAut(X) \cong \SAut(Y)$ must map $\egnets(X)$ onto $\egnets(Y)$. Note that this is true even if the EFP SFTs $X, Y$ are defined over different groups (though we do not have non-trivial examples where this happens).

We show that, under various conditions, the group $\gnets(X)$ is perfect, i.e.\ equal to its commutator: 

\begin{theorem}
Let $G$ be a residually finite countable group and let $X \subset \Sigma^G$ be an EFP SFT. Then we have $\gnets(X) = [\gnets(X), \gnets(X)]$ if one of the following holds:
\begin{itemize}
\item $G = \Z$ (Lemma~\ref{lem:ZPerfect});
\item $X$ has the property that for some finite shape $F$, in every $G \setminus F$ context the number of ways to fill the remaining $F$-hole is even (Lemma~\ref{lem:EvenPerfect});
\item $X$ is a full shift $\Sigma^G$, and $G$ has \emph{halvable finite-index subgroups}, meaning every finite index subgroup of $G$ has a subgroup of even index (Lemma~\ref{lem:HalvablePerfect}).
\end{itemize}
\end{theorem}

As a consequence of the first item, and some simple (classical) observations discussed in Section~\ref{sec:Inertness}, we obtain the following theorem, which in turn (combined with the above results) yields Theorem~\ref{thm:OneDimensionalSimple}:

\begin{theorem}
\label{thm:SIAutIsCommutator}
For any mixing one-dimensional SFT, we have $\SIAut(X) = \gnets(X) = \egnets(X)$.
\end{theorem}

Theorems~\ref{thm:CommutatorIsSimple} and~\ref{thm:SIAutIsCommutator} suggest that, in the absense of a standard definition of inertness on a general residually finite group $G$, we could take $\SIAut(X) = [\gnets(X), \gnets(X)]$ to be the definition of the inert subgroup of the stabilized automorphism group. Another obvious candidate for the inert subgroup is the commutator $[\SAut(X), \SAut(X)]$. However, as discussed in the introduction, already for $X$ a mixing one-dimensional SFT, $\SIAut(X)$ may be a proper subgroup of $[\SAut(X), \SAut(X)]$. 

In fact, we can go further: in the case of a mixing $\Z$-SFT, the subgroup of shift-commuting elements of $\egnets(X)$, i.e.\ $\Aut(X) \cap \egnets(X)$, coincides precisely with the inert subgroup of $\Aut(X)$. Thus, for a subshift (or at least an EFP SFT) on a general residually finite group, we may propose the group $\Aut(X) \cap \egnets(X)$ as a natural generalization of the inert automorphism group of $X$.

The precise statement of our main result, Theorem~\ref{thm:MainProof}, contains an extra generalization on top of  Theorem~\ref{thm:CommutatorIsSimple}, namely we can take any net of finite-index subgroups (with enough normal subgroups), and consider only the subgroups in this net in all the definitions. This does not change the proofs in any way, but it enlarges the class of groups to which our results apply.

\vspace{0.2cm}

We conclude with a few words of comparison between our proof, and the proof of Theorem~\ref{thm:FullShiftSimple} (i.e.\ simplicity of the inert subgroup in the full shift case) in \cite{HaKrSc21}. The starting point, and intuitively the most crucial observation, of these proofs is the same: any nontrivial normal subgroup of $\gnets(X)$ actually contains a nontrivial gate lattice. In \cite{HaKrSc21}, this statement is Lemma~5.2. Its proof is quite technical and spans over 30 pages.

Our proof of the analogous (but more general) claim can be summarized in one sentence that does not really hide any technicalities: commutatoring any nice enough homeomorphism (in particular any element of the stabilized automorphism group) by a gate we obtain a gate, and performing the same commutatoring on a sufficiently sparse finite-index subgroup gives the corresponding gate lattice. Translating this idea into a proof takes us a few pages as well, but most of this work is about setting up the algebra and basic theory of gate lattices.

\section{Setting the scene}
\label{sec:Definitions}

\subsection{Preliminaries}

In this section we fix conventions, recall some basic definitions, and introduce some new ones. By $A \Subset B$ we mean $A$ is a finite subset of $B$. Intervals $[i,j]$ with $i,j \in \Z$ are discrete, i.e.\ $[i,j] = \{k \in \Z \;|\; i \leq k \leq j\}$. By $\Sigma^*$ we denote finite (possibly empty) words over alphabet $\Sigma$, i.e.\ elements of the free monoid, and for $u \in \Sigma^*$, $|u|$ denotes the length of $u$. 
 
For groups, our commutator convention is $[a,b] = a^{-1}b^{-1}ab$, and conjugation is $a^b = b^{-1}ab$. The identity element of an abstract group $G$ is $e_G$. For groups of homeomorphisms (under function composition) the identity element is the identity map, written as $\id$. The notation $\langle \ldots \rangle$ means a group generated by $\ldots$. Here, $\ldots$ may include individual group elements, or sets of them (in which case we use elements of the set as generators).

In groups of homeomorphisms we write composition of $\phi_1$ and $\phi_2$ as $\phi_1 \circ \phi_2$ or simply $\phi_1 \phi_2$, and the rightmost homeomorphism is applied first. The \emph{support} of a homeomorphism is the smallest closed set such that every point outside is fixed.

A group $G$ is \emph{residually finite} if for every $g \neq e_G$, there exists a normal finite-index subgroup $H$ such that $g \notin H$. We say $G$ is a \emph{$2$-group} if every element has finite order which is a power of $2$. By $F_n$ we denote the free group on $n$ generators. By $S_n$ and $A_n$ we denote the symmetric and alternating group acting on $n$ elements, respectively, and $\Sym(A), \Alt(A)$ are the corresponding groups for a finite set $A$.

Throughout this paper (unless otherwise mentioned), $G$ denotes a countably infinite residually finite discrete group, not necessarily finitely-generated. We think of this as an ``ambient'' group, and often omit it in statements (unless we need to specify further properties of it).

We order finite subsets of a set $G$ by inclusion, and something holds for \emph{arbitrarily large} subsets of $G$ if for any finite set $F \Subset G$ it holds for some finite $S \supset F$. If $G$ is finitely-generated, we say a sequence of translated finite-index subgroups $H_ig_i$ gets \emph{arbitrarily sparse} if the distance between elements $hg_i, h'g_i$ is uniformly arbitrarily large for distinct $h, h' \in H_i$, with respect to some right-invariant word metric; equivalently, the word norm of the minimal non-identity element in $H_i$ grows without bound.

On subsets of a group $G$ we use the \emph{Fell topology}, namely the Cantor topology $\{0,1\}^G$ after identifying sets with their characteristic functions. For the subspace of subgroups this is also known as the \emph{Chabauty topology}. For a countable discrete group $G$, a net $H_i$ tends to the trivial group if for every finite subset $F$ of $G \setminus \{e_G\}$, eventually $e_G \in H_i$ and $H_i \cap F = \emptyset$. On a general group, $H_ig_i$ becomes arbitrarily sparse if $(H_i)_i$ tends to the trivial group in this topology.

We will frequently and without explicit mention use the following basic group theory fact, which works for any group $G$; for $K$ one can simply take the kernel of the translation action on left cosets of $H$.

\begin{lemma}
If $H \leq G$ is of finite index, then $H$ contains a normal subgroup $K \triangleleft G$ of finite index.
\end{lemma}

A \emph{topological dynamical system} is a pair $(X, G)$ where $X$ is a nonempty compact metrizable space, $G$ is a countable discrete group, and $G \curvearrowright X$ acts continuously on $X$ (of course, this action is also part of the data, but it is always clear from context). In the case $G = \Z$ we speak of \emph{one-dimensional systems}, and often write such a system as simply $(X, \sigma)$ where $\sigma$ is the homeomorphism corresponding to the cyclic generator $1 \in \Z$.

If $G$ is a countable infinite group and $\Sigma$ a finite set, we consider $\Sigma^G$ with the product topology; topologically it is just the Cantor set. The group $G$ acts on $\Sigma^G$ by $(g, x) \mapsto \sigma_g(x)$ where $\sigma_g(x)_h = x_{hg}$. To shorten formulas, we often omit ``$\sigma$'' and identify elements of $G$ with these translation maps. If $X$ is a topologically closed $G$-invariant set $X \subset \Sigma^G$, $(X, G)$ is called a \emph{subshift} and $x \in X$ is called a \emph{configuration}. When $X \subset \Sigma^G$, or $G$ is clear from context, we often write just $X$ for the subshift $(X, G)$.

For $X \subset \Sigma^G$ and $D \subset G$, for $x \in X$ write $x|D \in \Sigma^D$ for the partial configuration $\forall g \in D: y_g = x_g$ (we do not drop the right argument of $|$ to a subscript, to avoid complex formulas in subscripts and double subscripting). Partial configurations are also called \emph{patterns}, and they are \emph{finite} if their domain is. Write $X|D \subset \Sigma^D$ for the set of patterns $x|D$ where $x \in X$. In the one-dimensional situation for $u \in \Sigma^*$ we write $u \sqsubset X$ for $u \in X|[0,|u|-1]$, with the obvious identification of words and patterns. When a subshift is clear from context, $y \in \Sigma^D$ is a pattern, and $N \subset G$, define $\follow(y, N) = \{z \in \Sigma^N \;|\; \exists x \in X: x|D = y, x|N = z\}$. For two patterns $x \in \Sigma^D, y \in \Sigma^E$ with $D \cap E = \emptyset$, write $x \sqcup y$ for the obvious union pattern with domain $D \cup E$.

When a group $G$ is clear from context, we fix a net $(B_r)_r$ of finite symmetric (a set $B \subset G$ is \emph{symmetric} if $g \in B \iff g^{-1} \in B$) subsets of $G$ that exhausts it (we do not name the directed set of $r$s). We call the finite set $B_r$ the \emph{ball} of radius $r$. For a symmetric set $N$, write $A_{r,N} = NB_r \setminus B_r$ for the \emph{annulus} of \emph{thickness} $N$. As the notation may suggest, we like to pretend our groups are finitely-generated; in this case one may fix some finite set of generators (all results will be independent of this choice, more generally the net). For $d$ the corresponding right-invariant word metric, and $r \in \N$, we can pick $B_r = \{g \;|\; d(g, e_G) \leq r\}$. Picking $N = B_R$, the corresponding annulus is just $A_{r,N} = B_{r+R} \setminus B_r$, explaining the terminology. Patterns whose domain is an annulus $A_{r,N}$ are often called \emph{contexts}, and many of our arguments deal with the various fillings $x \in \follow(P, B_r)$ (or $x \in \follow(P, NB_r)$ if we want to keep the context visible) for contexts $P \in X|A_{r, N}$. Sometimes we use similar terminology with ``full'' contexts $x \in X|G \setminus N$ with $N \Subset G$.

The usual notion of isomorphism for topological dynamical systems is \emph{topological conjugacy}, meaning homeomorphism commuting with the group action. Up to topological conjugacy, we can characterize subshifts on countable groups as expansive actions on compact subsets of the Cantor set, where \emph{expansivity} means that there exists $\epsilon > 0$ such that
\[ (\forall g \in G: d(gx, gy) < \epsilon) \implies x = y \]
where $d$ metrizes the Cantor topology.

If $X \subset \Sigma^G$ is a subshift and $H \leq G$ is a subgroup, then by $(X, H)$ we refer to the topological dynamical system with space $X$ and action the restriction of the $G$-action to the subgroup $H$. When $H$ is of finite index in $G$, the system $(X, H)$ can itself be considered a subshift. Namely pick a set of left representatives $R$ and define $\phi : \Sigma^G \to (\Sigma^R)^H$ by $(\phi(x)_h)_r = x_{rg}$. Alternatively, in terms of the abstract characterization of subshifts, it is easy to see that passing to a subgroup of finite index only changes the expansivity constant, since the $G$-action is continuous. Even if $X \subset \Sigma^G$ is not necessarily $G$-invariant, we say it is a $K$-subshift if it is nonempty, closed and $K$-invariant.

A subshift of the form $\bigcap_{g \in G} gC$, where $C \subset \Sigma^G$ is clopen, is called an \emph{SFT} (short for subshift of finite type). If $P \in \Sigma^D$ where $D \Subset G$, the \emph{cylinder} defined by $P$ is $[P] = \{x \in X \;|\; x|D = P\}$. A clopen set is a finite union of cylinders, and a \emph{window} for an SFT $X$ is any symmetric finite set $N$ (meaning $N = \{g^{-1} \;|\; g \in N\}$) that contains all the sets $DD^{-1}$ such that $D$ is one of the supports of such cylinders. The important property of a window $N$ is the following; we omit the straightforward proof.

\begin{lemma}
If $X$ is SFT, and $N$ is a window for $X$, then for all $y \in X|A_{r,N}$, the choices of fillings $\follow(y, B_r)$ and $\follow(y, G \setminus NB_r)$ are independent in $X$, meaning for any $x \in \follow(y, B_r)$ and $z \in \follow(y, G \setminus NB_r)$, we have $x \sqcup y \sqcup z$ is in $X$.
\end{lemma}

We will need a simple fact about uniform convergence. Note that we only deal with homeomorphisms on Cantor space, so all metrics are equivalent and all continuous functions are uniformly continuous.

\begin{lemma}
\label{lem:UniformComposition}
Let $X, Y, Z$ be metric spaces and let $g_i, g: Y \to Z, f_i, f : X \to Y$ be functions, where $i$ runs over some directed set $\mathcal{I}$. If $g_i \rightarrow g$ and $f_i \rightarrow f$ uniformly and $g$ is uniformly continuous, then $g_i \circ f_i \rightarrow g \circ f$ uniformly.
\end{lemma}

\begin{proof}
Let $\epsilon > 0$ be arbitrary. Let $j_0$ be such that for $j \geq j_0$ we have $\forall y \in Y: d_Z(g_j(y), g(y)) < \epsilon/2$. Use uniform continuity of $g$ to find $\delta > 0$ such that $\forall x, y \in Y: d_Y(x, y) < \delta \implies d_Z(g(x), g(y)) < \epsilon/2$. Use uniform convergence of $(f_i)_i$ to find $i_0 \geq j_0$ such that for $i \geq i_0$ we have $\forall x \in X: d_Y(f_i(x), f(x)) < \delta$.

Suppose now that $i \geq i_0$ and let $x \in X$ be arbitrary. We have
\[ d_Z(g(f(x)), g(f_i(x))) < \epsilon/2 \] because $d_Y(f(x), f_i(x)) < \delta$ and by the choice of $\delta$, and
we have
\[ d_Z(g(f_i(x)), g_i(f_i(x))) < \epsilon/2 \]
because $i \geq j_0$, so by the triangle inequality we have
\[ d_Z(g(f(x)), g_i(f_i(x))) \leq d_Z(g(f(x)), g(f_i(x))) + d_Z(g(f_i(x)), g_i(f_i(x))) < \epsilon \]
proving uniform convergence of $g_i \circ f_i$ to $g \circ f$.
\end{proof}

\subsection{Eventual filling property and many fillings property}

\begin{definition}
\label{def:EFP}
A subshift $X \subset \Sigma^G$ has the \emph{eventual filling property}, or \emph{EFP}, if
\[ \forall F \Subset G: \exists N \Subset G: \forall x, y \in X: \exists z \in X: z|F = x|F \wedge z|(G \setminus N) = y|(G \setminus N). \]
\end{definition}

EFP is a \emph{gluing property}, meaning it deals with compatibility of patterns on different areas of the group. On the groups $\Z^d$, the EFP property can be seen as a weakening of the uniform filling property, introduced in \cite{RoSa99}. To the best of our knowledge it has not been studied previously even on $\Z^2$.

As a technical weaker notion, we say a subshift has the \emph{many fillings property}, or \emph{MFP}, if as $F \Subset G$ tends to $G$, the number of configurations agreeing with $x|{G \setminus F}$ tends to infinity uniformly in $x \in X$. A subshift is \emph{nontrivial} if it has at least two configurations. The following is a simple exercise (recalling that a closed nonempty subset of Cantor space without isolated points is homeomorphic to Cantor space).

\begin{lemma}
\label{lem:EFPMFPCantor}
Every nontrivial EFP subshift (on an infinite group) has MFP. Every MFP subshift is homeomorphic to Cantor space.
\end{lemma}

An \emph{automorphism} of a subshift $X \subset \Sigma^G$ is a homeomorphism $f : X \to X$ that commutes with the action of $G$. By the definition of the topology on $\Sigma^G$, an automorphism has a (not necessarily unique) finite \emph{neighborhood} $N \subset G$ and a \emph{local rule} $\hat f : \Sigma^N \to \Sigma$ such that $f(x)_g = \hat f(P)$ where $P \in \Sigma^N$ is defined by $P_h = x_{hg}$. This is a special case of the so-called \emph{Curtis-Hedlund-Lyndon theorem} \cite[Theorem~1.8.1]{CeCo10}.\footnote{The theorem is more generally about shift-commuting continuous functions, but here we restrict to the bijective case.} Obviously the inverse of an automorphism is an automorphism as well, so the automorphisms of a subshift form a group denoted by $\Aut(X, G)$. A topological conjugacy between two subshifts is similarly defined by a local rule.

\begin{lemma}
\label{lem:EFPIsEquivalent}
For one-dimensional SFTs, EFP is equivalent to
\[ \exists m: \forall u, v \sqsubset X: \exists w: |w| = m \wedge uwv \sqsubset X. \]
\end{lemma}

(This is the one-dimensional case of the uniform filling property of \cite{RoSa99}.)

\begin{proof}
If the condition holds, then EFP is clear (even without the SFT assumption), namely for any finite set $F \subset [i,j]$ we can pick $N = [i-m, j+m]$ and use the condition on both sides of the interval to glue $F$-patterns (extended arbitrarily to $[i,j]$) to $(\Z \setminus N)$-patterns.

From Curtis-Hedlund-Lyndon, we easily see that topological conjugacy preserves both EFP and the condition in the lemma (even every \emph{factor map}, i.e.\ surjective $G$-commuting continuous map, preserves them), so we can conjugate $X$ to an \emph{edge shift} \cite[Theorem~2.3.2]{LiMa95}, meaning $X$ is the set of paths (sequences of edges with matching endpoints) in a finite directed graph $(V,E)$. Let $M$ be the matrix with $M_{a,b}$ the number of edges from vertex $a$ to vertex $b$

Suppose now that we have EFP. Pick $F = \{0\}$ and $N$ as in the definition of EFP. We can clearly make $G \setminus N$ smaller without breaking the gluing property for this pair, so we may take $N = [i, j]$ with $i \leq 0 \leq j$. Now in particular, by applying the gluing property to all pairs of particular configurations and looking at the rightmost vertex of the edge at $0$, and the leftmost vertex of the edge at $j+1$, we see that for any two vertices $a,b \in V$, we have a path of length $j$ from $a$ to $b$. Thus $M^j$ is a matrix with all entries positive. Now the condition of the lemma holds with $m = j$.
\end{proof}

The condition in the lemma is known to be equivalent to topological mixing \cite[Exercise~6.3.5]{LiMa95}, and we refer to it as simply \emph{mixing}. In the proof we showed that EFP implies that $X$ is defined, as an edge shift, by a \emph{primitive} matrix, namely one with a positive power.

\subsection{A dual notion of continuity}

A crucial property of automorphisms of subshifts is that they satisfy a kind of ``dual version of continuity''. Namely, continuity of a map $f : X \to X$ for a subshift $X \subset \Sigma^G$ (and more generally for $X$ a closed subset of Cantor space) means that information cannot move from the ends of $G$ to near the origin when $f$ is applied, but rather the new value at a particular position is only a function of nearby values. Our dual notion is that information cannot move from near the origin of $G$ to its ends in one step. 

\begin{definition}
Let $G$ be a countable set, $\Sigma$ an alphabet and $X \subset \Sigma^G$. A homeomorphism $f : X \to X$ is \emph{ntinuous} if for each finite set $S$ there exists a finite set $F$ such that for each $g \notin F$, the map $x \mapsto f(x)_g$ factors through the projection $x \mapsto x|{G \setminus S}$. A homeomorphism is \emph{bintinuous} if it is ntinuous, and its inverse is also ntinuous.
\end{definition}

Not every ntinuous homeomorphism is bintinuous, one example is the map $f : \{0,1\}^\N \to \{0,1\}^\N$  defined by $f(x)_0 = x_0$ and $\forall i \geq 1: f(x)_i \equiv x_{i-1} + x_i \bmod 2$.

The following is easy to show.

\begin{lemma}
The bintinuous homeomorphisms on a subshift $X$ form a group.
\end{lemma}

\section{Gates, gate lattices and the stabilized automorphism group}

\subsection{Gates}
\label{sec:Gates}

\begin{definition}
A \emph{gate} on a subshift $X \subset \Sigma^G$ is a homeomorphism $\chi : X \to X$ such that for some $N \Subset G$ we have $\chi(x)_g = x_g$ for all $g \notin N$. Such an $N$ is called a \emph{weak neighborhood}. Write $\gates$ for the group generated by gates.
\end{definition}

Note that ``$\gates$'' is a fancy ``G'', and stands for ``gate''. Of course, $\gates$ depends on $X$, but omitting $X$ in the notation should not cause confusion as we rarely need to consider two subshifts simultaneously.

The following lemma was proved for $G = \Z$ in \cite{Sa22a}; the proof for general $G$ is the same, and follows more or less directly from the definition of the product topology.

\begin{lemma}
\label{lem:StrongWeak}
Let $X \subset \Sigma^G$ be a subshift. A homeomorphism $\chi : X \to X$ is a gate if and only if it admits some $N \Subset G$ and $\hat\chi : X|N \to X|N$ such that for all $g \notin N$ we have $\chi(x)_g = x_g$ for all $x \in X$, and for all $g \in N$ we have $\chi(x)_g = \hat\chi(x|N)_g$.
\end{lemma}

A set $N$ as in the lemma is called a \emph{strong neighborhood}, and $\hat \chi$ a \emph{local rule}.

It is clear that $\gates$ in fact consists of gates, as we can always increase the strong neighborhoods of two gates to be equal (after which they compose like permutations). For a gate $\chi$ and $g \in G$ write $\chi^g = \sigma_{g^{-1}} \circ \chi \circ \sigma_g$. Note in particular that when $G = \Z$, $\chi^n$ does not refer to iteration -- we never need to iterate a gate. One can see $\chi^g$ as applying the gate ``at'' $g$, if we use the convention where configurations of $X$ are seen as vertex-labelings of some left Cayley graph of the group $G$ (at least when it is finitely-generated). If $\chi$ has strong (resp.\ weak) neighborhood $N$, then $\chi^g$ has strong (resp.\ weak) neighborhood $Ng$.


\begin{lemma}
\label{lem:GatesNormalDerp}
Let $X \subset \Sigma^G$ be a subshift. Let $f : X \to X$ be a homeomorphism with ntinuous inverse and $\chi$ be any gate on $X$. Then $\chi^f$ is a gate as well. Furthermore, if $f$ is bintinous, the map $\chi \mapsto \chi^f$ is an automorphism of $\gates$. 
\end{lemma}

\begin{proof}
Suppose that $N \Subset G$ is a strong neighborhood for $\chi$. 
By ntinuity of $f^{-1}$, outside some finite set $F$, the image of $f^{-1}$ can be determined without looking at the cells in $N$; in other words for $g \notin F$, we have $f^{-1}(\chi(f(x)))_g = f^{-1}(f(x))_g = x_g$ since the application of $\chi$ does not affect the $f^{-1}$-image. This shows that $F$ is a weak neighborhood for $\chi^f$, and thus $\chi^f$ is a gate.

For the second claim, for a bintinuous homeomorphism $f$, the inverse of $\chi \mapsto \chi^f$ is $\chi \mapsto \chi^{f^{-1}}$, where $f^{-1}$ is also a bintinuous homeomorphism. Thus both $f$ and $f^{-1}$ act bijectively on $\gates$ by conjugation. They are inner automorphisms of the group of all homeomorphisms of $X$, so in particular their restrictions to $\gates$ give automorphisms.
\end{proof}

\begin{lemma}
\label{lem:GatesNormal}
Let $X \subset \Sigma^G$ be a subshift. The group $\gates$ is normalized by the group of bintinuous homeomorphisms on $X$. 
\end{lemma}

\begin{proof}
By the above lemma, conjugation by bintinuous homeomorphisms fixes $\gates$, which is what normalization means. 
\end{proof}

Say a gate is \emph{eventually even} if for all large enough strong neighborhoods $N$ the corresponding permutation $\pi \in \Sym(X|N)$ restricted to any complement pattern is even, i.e.\ for any $y \in X|{G \setminus N}$, the restriction of $\pi$ to the set $\follow(y, N)$ is even. We call such $N$ \emph{even neighboods}. Say a gate is \emph{sometimes even} if the same is true for at least one strong neighborhood $N$.

\begin{lemma}
\label{lem:CommutatorCharacterization}
On every subshift, each of the following implies the next.
\begin{enumerate}
\item $\chi = [\chi_1, \chi_2]$ for some gates $\chi_1, \chi_2$;
\item $\chi \in [\gates, \gates]$;
\item $\chi$ is eventually even;
\item $\chi$ is sometimes even.
\end{enumerate}
The last two items are always equivalent, and on an MFP SFT all four are equivalent.
\end{lemma}

\begin{proof}
The implication (1) $\implies$ (2) is trivial. For (2) $\implies$ (3) take any $N$ larger than the radius of all the gates involved in a composition of commutators of gates. For any $y \in X|{G \setminus N}$, the permutation of $\follow(y, N)$ performed in the context is just the corresponding composition of commutators of gates restricted to $\follow(y, N)$, and thus is in the commutator subgroup of the symmetric group of that set, which is the corresponding alternating group. The implication (3) $\implies$ (4) is trivial.

We show (4) $\implies$ (3) in every subshift, so (3) and (4) are equivalent. Simply observe that if $N$ is a strong neighborhood such that $\chi$ performs an even permutation on $N$ in every $(G \setminus N)$-context, then the same is true for any $(G \setminus N')$-context for $N' \supset N$, as for any $(G \setminus N')$-context we can write the permutation $\chi$ performs on $\follow(y, N')$ as a finite composition of even permutions. Namely, for each of the finitely many extensions $z \in \follow(y, G \setminus N)$ the permutation $\chi$ performs on the pattern in $N$ is even by assumption.


It now suffices to show that in an MFP SFT $X$, (3) $\implies$ (1). To see this, let $N$ be a window for $X$ and pick a large strong neighborhood $B_r$ such that there are at least $5$ fillings of each $A_{r,N}$-context. Now increasing the strong neighborhood to $NB_r$, we have a permutation of $X|NB_r$ which does not modify the contents of the annulus $A_{r,N}$ and for each pattern on the annulus performs an even permutation on the pattern inside. Since there are at least $5$ extensions of the pattern and $[S_n,S_n] = \{[g, h] \;|\; g, h \in S_n\}$ for $n \geq 5$ \cite{Or51}, we can write the restriction to each $A_{r,N}$-context as a commutator of two permutations in that same context. For different contexts the permutations commute, so we can write this as a commutator of two gates.
\end{proof}

Due to the last sentence of the previous lemma, we simply call eventually/sometimes even gates \emph{even} (as we are only interested in MFP SFTs here). Write $\egates$ for the group generated by even gates. We point out some immediate corollaries of the above lemma.

\begin{lemma}
\label{lem:CommutatorCharacterization2}
If $X$ is an MFP SFT, then:
\begin{itemize}
\item $\egates = [\gates, \gates]$,
\item $\egates$ is the set of even gates,
\item the commutator width of $\gates$ is $1$,
\item if $\chi \in \egates$, then $\chi^f \in \egates$ for $f$ any bintinuous homeomorphism, and
\item $\egates$ is normalized by the group of bintinuous homeomorphisms.
\end{itemize}
\end{lemma}

\begin{proof}
The first item is the equivalence of items $2$ and $3$ in the previous lemma. The second item follows from item $3$ of the previous lemma: for any even gates $\chi_1, \chi_2$ we can pick a common strong neighborhood, after which it is clear that their composition is also (eventually) even. The third item follows from the item $1$ of the previous lemma. The fourth item follows because the commutator subgroup is characteristic, and $\chi \mapsto \chi^f$ is an automorphism of $\gates$ by Lemma~\ref{lem:GatesNormalDerp} and Lemma~\ref{lem:AutBint}. The fifth is an immediate consequence of the fourth. 
\end{proof}

\begin{example}
It is possible that $\chi$ is even, $N$ is a common strong neighborhood for $\chi$ and $\chi^f$, and $N$ is an even neighborhood for $\chi$, yet $N$ is not an even neighborhood for $\chi^f$. Namely, pick $G = \Z$, and consider $\chi$ and $\chi^{\sigma}$.

For this, pick the matrix $M = \left(\begin{smallmatrix} 1 & 1 \\ 1 & 1 \end{smallmatrix}\right)$ and let $X$ be the corresponding edge shift, i.e.\ at each $n \in \Z+\frac12$ we have a vertex, and at each $n \in \Z$ we have an edge. Let us call the vertices $\{a, b\}$, and note that $M$ simply says we have a unique edges between each pair of vertices. Now pick $\chi$ the unique gate with strong neighborhood $\{0,1\}$ (so $\chi$ sees two edges) that permutes the word nontrivially if and only if the vertex at $-\frac12$ is $a$. In cycle notation, the action on the tuple of vertices at $(-\frac12, \frac12, \frac32)$ is
\[ (aaa \; aba)(aab \; abb)(baa)(bab)(bba)(bbb). \]

Note that $\chi^{\sigma}$ performs the same modification on the word (consisting of edges) appearing in $\{1,2\}$. Now consider these with strong neighborhood $\{0,1,2\}$. Obviously $\chi$ is now an even gate, as in any context where the vertex at $-\frac12$ is $a$ we perform two swaps, and in any context where it is $b$, we do nothing. On the other hand, $\chi^{\sigma}$ always performs an odd permutation, since to get a nontrivial action we must pick the vertex at $\frac12$ to be $a$.

Nevertheless, since $\egates = [\gates, \gates]$, $\chi^\sigma$ must be even. Indeed it is: by shift-symmetry, $\chi$ and $\chi^\sigma$ have the same evenness if computed with neighborhoods $N$ and $N+1$ respectively. \qee
\end{example}



\subsection{Gate lattices}
\label{sec:GateLattices}

We say two gates $\chi, \chi'$ \emph{commute} if they do, i.e.\ if $[\chi, \chi'] = \chi^{-1} \circ (\chi')^{-1} \circ \chi \circ \chi' = \id$. It is clear that if $\chi, \chi'$ have strong radii $N, N'$ respectively, and $N \cap N' = \emptyset$, then $\chi$ and $\chi'$ commute. If $S \subset G$ is a (possibly infinite) subset, and $\chi^s$ commutes with $\chi^t$ for all $s, t \in S$, then we say \emph{the product $\prod_{s \in S} \chi^s$ commutes}, or just that \emph{$\chi^S$ commutes}, and define
\[ \chi^S(x) = \lim_{F \Subset S} (\prod_{k \in F} \chi^k)(x) \]
by taking the pointwise limit, when this limit exists (note that $\prod$ means function composition here). We show that it always exists.

\begin{lemma}
\label{lem:PointwiseLimit}
On any subshift, if $\chi^K$ commutes, then the pointwise limit is well-defined, and convergence is uniform.
\end{lemma}

\begin{proof}
To see that the pointwise limit is well-defined, observe that in a finite subproduct $\chi^F$, $F \Subset K$, the value of $\chi^F(x)_g$ only depends on the value of finitely many elements of $G$, since we may order the product so that translates of $\chi$ that may change the value of $g$ (i.e.\ $g$ is in their strong neighborhood) are applied first. The same argument gives uniform convergence: the value at $g$ stabilizes once we have applied all translates of $\chi$ with $g$ in their strong neighborhood.
\end{proof}

\begin{lemma}
\label{lem:ExpHIsAuto}
Suppose that $X \subset \Sigma^G$ is a subshift, $K \leq G$ is a subgroup, $\chi$ is a gate on $X$ and $\chi^K$ commutes. Then $\chi^K$ is an automorphism of $(X,K)$.
\end{lemma}

\begin{proof}
To see that $\chi^K$ is a homeomorphism, observe that by Lemma~\ref{lem:PointwiseLimit} it is a uniform limit of continuous functions, thus continuous (alternatively, the proof shows this directly). It clearly has continuous inverse $(\chi^{-1})^K$ (it is easy to show that this product commutes when $\chi^K$ does), thus it is a homeomorphism.

We check commutation with $K$-shifts. If $h \in K$, then
\begin{align*}
\chi^K \circ \sigma_h &= \lim_{F \Subset K} (\prod_{k \in F} \chi^k) \circ \sigma_h \\
&= \lim_{F \Subset K} \prod_{k \in F} \sigma_{k^{-1}} \circ \chi \circ \sigma_k \circ \sigma_h \\
&= \lim_{F \Subset K} \prod_{k \in F} \sigma_h \circ \sigma_{(kh)^{-1}} \circ \chi \circ \sigma_{kh} \\
&= \lim_{F \Subset K} \prod_{k \in F} \sigma_h \circ \sigma_{k^{-1}} \circ \chi \circ \sigma_{k} \\
&= \sigma_h \circ \chi^K
\end{align*}
where uniform convergence of products means limits commute with composition (Lemma~\ref{lem:UniformComposition}), and the fourth equality holds because $Fh$ runs over the finite subsets of $K$ as $F$ does. 
\end{proof}

\begin{lemma}
\label{lem:LatticeBase}
Suppose $X \subset \Sigma^G$ is a subshift, $H \leq G$, and $\chi$ is a gate on $X$. If $\chi^{Hg'}$ commutes for some $g' \in G$, then for all $g \in G$, $\chi^{Hg}$ commutes. If $g' = e_G$, then $\chi^{Hg} = (\chi^H)^{g}$.
\end{lemma}

In words, the first claim of the lemma says that if $\chi^S$ commutes for one right coset $S$ of $H$, then it commutes for every right coset $S$ of $H$.

\begin{proof}
Let $S = Hg'$ so $Hg = Sk$ where $k = (g')^{-1} g$. Interpreting infinite products as pointwise uniform limits of finite products and applying Lemma~\ref{lem:UniformComposition} to pull functions out of the limit, the calculation
\[ \chi^{Sk} = \prod_{h \in S} \sigma_{k^{-1}h^{-1}} \circ \chi \circ \sigma_{hk} = (\prod_{h \in S} \sigma_h \circ \chi \circ \sigma_{h^{-1}})^k = (\chi^S)^k \]
shows the commutation of $\chi^{Sk}$. If $g' = e_G$, then we get further
\[ \chi^{Hg} = \chi^{Sk} = (\chi^S)^k = (\chi^H)^g \]
which is the last formula.
\end{proof}

\begin{example}
We note that the product $\chi^{gH}$ may not commute even if $\chi^H$ commutes. Suppose e.g.\ that $G = F_2 = \langle a, b \rangle$ (the free group on two free generators $a, b$), $X = \{0,1\}^G$, $H = \langle a \rangle$ and $\chi$ swaps the symbols at $\{b, ba\}$. Clearly $\chi^H$ commutes, but $\chi^{b^{-1}H}$ does not. \qee
\end{example}

\begin{definition}
\label{def:Gnets}
Let $X \subset \Sigma^G$ be a subshift on a residually finite group. Maps of the form $\chi^{Hg}$, where $\chi$ is a gate on $X$, $H \leq G$ is of finite index and $\chi^{Hg}$ commutes, are called \emph{gate lattices}, and they are \emph{even} when $\chi$ is. Write $\gnets$ for the group generated by gate lattices, and $\egnets$ for the group generated by even gate lattices.
\end{definition}

Note that ``$\gnets$'' is a fancy ``L'', and stands for ``lattice'', which refers to the fact gates are applied at the points of a (translated) lattice. As with $\gates$ (and $\egates$), the groups $\gnets$ and $\egnets$ depend on the subshift $X$, but this should not cause any confusion. 

\begin{lemma}
\label{lem:Decomposition}
Let $X$ be a subshift, and $\chi$ a gate on $X$. Suppose $\chi^{Hg}$ commutes, and $K \leq H$ is of finite index. Then $\chi^{Hg}$ can be written as a finite commuting product of commuting gates of the form $\chi^{Kg'}$.
\end{lemma}

\begin{proof}
Simply write $H = \bigcup_{h \in T} Kh$ for some set of representatives $T$, and observe that
\[ \prod_{h \in H} \chi^{hg} = \prod_{h \in T} \prod_{k \in K} \chi^{khg} = \prod_{h \in T} \chi^{Khg} \]
by the commutation of $\chi^{Hg}$.
\end{proof}

\begin{lemma}
\label{lem:NormalAuto}
If $X \subset \Sigma^G$ is a subshift, $\chi$ is a gate on $X$, and $H$ is a normal subgroup of $G$, then $\chi^{Hg}$ is an automorphism of $(X, H)$ whenever $\chi^H$ commutes.
\end{lemma}

\begin{proof}
We have $\chi^{H g} = \chi^{g H}$ by normality, and the latter product, equal to $(\chi^g)^H$, must commute since by  Lemma~\ref{lem:LatticeBase} $\chi^{H g}$ does, and $g H$ and $H g$ are literally the same subset of $G$ where $\chi$ gets applied. Now $\chi^{Hg} = (\chi^g)^H$ is an automorphism of $(X, H)$ by Lemma~\ref{lem:ExpHIsAuto}. 
\end{proof}

We recall a definition of Hartman-Kra-Schmieding (generalized to residually finite acting groups, as seems appropriate here).

\begin{definition}
\label{def:StabAut}
Let $G$ be a residually finite group, and $X \subset \Sigma^G$ a subshift. The \emph{stabilized automorphism group} $\SAut(X, G)$ is the union of $\Aut(X, H)$ where $H$ ranges over finite-index subgroups of $G$.
\end{definition}

By union we mean here simply the set-theoretic union. The union is increasing along the net of finite-index subgroups, so this is indeed a group. It is of course the same as the direct union of the groups $\Aut(X, H)$ along the same net. We sometimes call elements of the stabilized automorphism group \emph{stabilized automorphisms}.

The following is clear from the general Curtis-Hedlund-Lyndon theorem discussed above Lemma~\ref{lem:EFPIsEquivalent}, observing that stabilized automorphisms are automorphisms for the subaction of a finite-index subgroup (which is itself a subshift).

\begin{lemma}
\label{lem:AutBint}
Stabilized automorphisms of subshifts are bintinuous.
\end{lemma}

In principle, there is no need to restrict to residually finite groups in Definition~\ref{def:StabAut}, but in this case the group can be of rather different nature, e.g.\ on a group $G$ with no finite index subgroups $\SAut(X)$ would be just $\Aut(X)$. Our results also do not generalize to general groups, thus we make this assumption throughout.

We note that $\SAut(X)$ itself is not typically residually finite. Indeed, even if $X$ is a non-trivial mixing SFT on $\Z$, $\SAut(X)$ is never residually finite (although it is the union of the groups $\Aut(X, \sigma^n)$ which are residually finite).

\begin{lemma}
For any subshift $X$, the groups $\egnets$ and $\gnets$ are contained in the stabilized automorphism group of $X$. Indeed $\egnets \leq \gnets \leq \SAut(X)$.
\end{lemma}

\begin{proof}
Every finite-index subgroup $H$ of contains a normal finite-index subgroup $K \triangleleft G$. By Lemma~\ref{lem:Decomposition} we can write any $\chi^{Hg}$ as a composition of maps of the form $\chi^{Kg'}$, and by Lemma~\ref{lem:NormalAuto} these are automorphisms of $(X, K)$, thus in the stabilized automorphism group.
\end{proof}

We note that, by the proof of Lemma~\ref{lem:NormalAuto}, we could in principle eliminate the right translates in the definition of a gate, without changing the group of gate lattices.\footnote{While this simplifies the definition, sticking to such gate lattices would require us to change our gates to their shift conjugates after every application of Lemma~\ref{lem:Decomposition}, which would complicate many proofs.}

\begin{lemma}
\label{lem:NoTranslateA}
The group $\gnets$ is generated by gates of the form $\chi^K$ where $K$ is a normal finite-index subgroup. The same is true for $\egnets$.
\end{lemma}

\begin{proof}
By Lemma~\ref{lem:Decomposition}, elements of the form $\chi^{Kg}$ with $K$ normal are a generating set. As in the proof of Lemma~\ref{lem:NormalAuto}, we have $\chi^{Kg} = \chi^{gK}$. We have $\chi^{gK} = (\chi^g)^K$ by definition, and $\chi^g$ is a gate.
\end{proof}

\subsection{Lattice nets}
\label{sec:LatticeNets}

We do not need to take \emph{all} subgroups of finite index in the definitions of $\egnets, \gnets, \SAut$, for some of our results. In particular in our main result Theorem~\ref{thm:MainProof} we allow the subgroups $\mathcal{I}$ to be essentially arbitrary. This gives results not only for $\SAut(X)$, but for a large collection of its subgroups. Readers only interested in results for $\SAut(X)$ can simply always take $\mathcal{I}$ to be the set of all finite-index subgroups, and skip this section.

\begin{definition}
Let $\mathcal{I}$ be a set of finite-index subgroups of $G$. Ordering $\mathcal{I}$ by reverse inclusion, we also think of it as a directed set, and $I \mapsto I : \mathcal{I} \to \mathcal{I}$ can be seen as a net of subgroups. 
We say $\mathcal{I}$ is a \emph{lattice net} if it tends to the trivial subgroup in the Chaubauty topology, and for all subgroups $H \in \mathcal{I}$, there is also a normal subgroup $K \triangleleft G$ with $K \in \mathcal{I}$ such that $K \leq H$.
\end{definition}


Note that for a lattice net to even exist, $G$ has to be residually finite. A subset $\mathcal{J} \subset \mathcal{I}$ of a directed set $\mathcal{I}$ is \emph{cofinal} if for all $I \in \mathcal{I}$, $J \geq I$ for some $J \in \mathcal{J}$. The definition of a lattice net simply states that normal subgroups (that appear in the net) must be cofinal in the net; we do not, however, need that all normal subgroups of $G$ appear in the net. Since all our definitions will only depend on the cofinality class of the lattice net, by the last condition above we could equivalently restrict our lattice nets to contain only normal subgroups. The following is proved exactly like Lemma~\ref{lem:NoTranslateA}.

\begin{lemma}
\label{lem:NoTranslate}
For any lattice net $\mathcal{I}$, the group $\gnets_{\mathcal{I}}$ is generated by gates of the form $\chi^K$ where $K$ is a normal subgroup in $\mathcal{I}$. The same is true for $\egnets_{\mathcal{I}}$.
\end{lemma}

If $\mathcal{I}$ is a lattice net, we can generalize the stabilized automorphism group in an obvious way to $\SAut(X, \mathcal{I}) = \bigcup_{H \in \mathcal{I}} \Aut(X, H)$. We also define $\mathcal{I}$-gate lattices in the obvious way, as well as notations $\gnets_{\mathcal{I}}, \egnets_{\mathcal{I}}$. Note that of course $\gnets_{\mathcal{I}}, \egnets_{\mathcal{I}}$ are respectively subgroups of $\gnets, \egnets$, and the definitions of $\mathcal{I}$-less $\SAut, \gnets, \egnets$ correspond precisely to the case where we take $\mathcal{I}$ the set of all finite-index subgroups.

In the remainder of this section we show that considering general lattice nets indeed adds generality. For two lattice nets $\mathcal{J}, \mathcal{I}$, we say that $\mathcal{J}$ is \emph{far from cofinal} in $\mathcal{I}$ if for all $N \Subset G$, there exists $H \in \mathcal{I}$ such that for all $K \in \mathcal{J}$, there exists $k \in K$ such that $Nk \cap NH = \emptyset$. Note that for $N = \{e_G\}$, this is exactly the complement of cofinality.

\begin{proposition}
\label{prop:IsMoreGeneral}
Suppose $X \subset \Sigma^G$ is an MFP SFT and $\mathcal{I}$, $\mathcal{J}$ are two lattice nets. Then
\begin{itemize}
\item if $\mathcal{J}$ is cofinal in $\mathcal{I}$, then
\[ \SAut(X, \mathcal{I}) \leq \SAut(X, \mathcal{J}), \; \gnets_{\mathcal{I}} \leq \gnets_{\mathcal{J}}, \;\mbox{and} \; \egnets_{\mathcal{I}} \leq \egnets_{\mathcal{J}}; \]
\item if $\mathcal{J}$ is far from cofinal in $\mathcal{I}$, then
\[ \SAut(X, \mathcal{I}) \not\leq \SAut(X, \mathcal{J}), \; \gnets_{\mathcal{I}} \not\leq \gnets_{\mathcal{J}}, \; \mbox{and} \; \egnets_{\mathcal{I}} \not\leq \egnets_{\mathcal{J}}. \]
\end{itemize}
\end{proposition}

\begin{proof}
We first show the claims for $\egnets$. If $\mathcal{J}$ is cofinal in $\mathcal{I}$, then Lemma~\ref{lem:Decomposition} shows that $\egnets_{\mathcal{I}} \leq \egnets_{\mathcal{J}}$. (This is true for any subshift.)

Suppose now that $\egnets_{\mathcal{I}} \leq \egnets_{\mathcal{J}}$. Let $\chi$ be a nontrivial even gate on $X$ with some strong radius $N$. One clearly exists in an MFP SFT, simply exchange three fillings of a large thick annulus.

Consider any gate lattice $\chi^H \in \gnets_{\mathcal{I}}$ where $H \in \mathcal{I}$ is normal. If $\chi^H \in \gnets_{\mathcal{J}}$, then in particular by Lemma~\ref{lem:Decomposition} there is a finite-index normal subgroup $K \triangleleft G$ with $K \in \mathcal{J}$ such that
\[ \chi^H = \chi_1^{K k_1} \circ \cdots \circ \chi_n^{K k_n}. \]
where each $\chi_i^K$ commutes by Lemma~\ref{lem:LatticeBase}.

Now Lemma~\ref{lem:NormalAuto} shows that each $\chi_i^{K k_i}$ is an automorphism of $(X, K)$, so $\chi^H$ commutes with the $K$-action. Again by Lemma~\ref{lem:LatticeBase},
$(\chi^H)^k = \chi^{Hk} = \chi^{kH}$,
and we obtain that $\chi^H = \chi^{kH}$ whenever $k \in K$. In particular, $NkH \cap NH \neq \emptyset$ for all $k \in K$, since $\chi^{kH}$ performs a nontrivial rewrite of the contents of $NkH$ in its input configuration. This implies $Nk \cap NH \neq \emptyset$ for all $k \in K$. Since $H$ was arbitrary, this contradicts far-from-cofinality.

For $\gnets$, the same proof works verbatim. For $\SAut$, the proof is easier: $\SAut(X, \mathcal{I}) \leq \SAut(X, \mathcal{J})$ is immediate from the definition, and for the other direction we construct $\chi^H \in \SAut(X, \mathcal{I})$ exactly as above, and now $\chi^H$ commutes with the $K$-action for some $K \in \mathcal{J}$ directly by the definition of $\SAut(X, \mathcal{J})$.
\end{proof}

Alternatively, it should be possible to replace ``far from cofinality'' with some assumptions on symmetry-breaking possibilities of gates (or perhaps even remove the assumption completely). We show that the above at least takes care of $\Z^d$.

\begin{lemma}
\label{lem:Zdcofinality}
Let $\mathcal{I}, \mathcal{J}$ be lattice nets in $\Z^d$. Then either $\mathcal{J}$ is cofinal in $\mathcal{I}$, or $\mathcal{J}$ is far from cofinal in $\mathcal{I}$.
\end{lemma}

\begin{proof}
Suppose $H \in \mathcal{I}$ is such that no $K \in \mathcal{J}$ is contained in $H$. Let $N \Subset H$ be given, and suppose $N$ is contained in the Euclidean ball of radius $\lfloor r/2 \rfloor$. Refine $H$ (i.e.\ move further into $\mathcal{I}$) so that it contains no nonzero vector of length less than $4 r$, and let $K \in \mathcal{J}$ be arbitrary. Let $\vec v \in K$ be any vector not contained in $H$ (which exists by assumption).

Suppose that $d(\vec v, \vec u) = t \leq r$ for some $\vec u \in H$. Let $\vec w = \vec v - \vec u$ and observe that
\[ r \leq d(\vec v + \lceil r/t \rceil \vec w, \vec u) = t + \lceil r/t \rceil t \leq 3 r \]
so in fact $d(\vec v + \lceil r/t \rceil \vec w, H) \geq r$, using that no two vectors in $H$ are at distance less than $4r$.

The above shows that that $H + K$ contains a vector at distance $\geq r$ from $H$. Now if $\vec u \in H, \vec v \in K$ satisfy $d(h + k, H) \geq r$, then clearly also $d(k, H) \geq r$. It follows that $k+N \cap H+N = \emptyset$, proving far from cofinality.
\end{proof}

In the case of $\Z^d$, the previous proposition and lemma together show that $\mathcal{J}$ is cofinal in $\mathcal{I}$ if and only if $\gnets_{\mathcal{I}} \leq \gnets_{\mathcal{J}}$, and $\gnets_{\mathcal{I}} = \gnets_{\mathcal{J}}$ characterizes cofinal equivalence; same for $\egnets$ and $\SAut$. The group $\Z$ already has many non-cofinal pairs of lattice nets, for example the lattice nets $(p^n\Z)_{n \in \N}$ are pairwise non-cofinal for distinct primes $p$.

\subsection{Evenness of gate lattices}
\label{sec:AllEven}

\begin{lemma}
\label{lem:egnetsIsCommu}
If $X$ is an MFP SFT on a residually finite countably infinite group $G$, and $\mathcal{I}$ any lattice net, then $\egnets_{\mathcal{I}} = [\gnets_{\mathcal{I}}, \gnets_{\mathcal{I}}]$. 
\end{lemma}


\begin{proof}
First, let us show $\egnets_{\mathcal{I}} \subset [\gnets_{\mathcal{I}}, \gnets_{\mathcal{I}}]$. Let $\chi^{Hg} \in \egnets_{\mathcal{I}}$. By Lemma~\ref{lem:CommutatorCharacterization} we have $\chi = [\chi_1, \chi_2]$ for some gates $\chi_1, \chi_2$. Using Lemma~\ref{lem:Decomposition} and taking a sufficiently sparse subgroup in $\mathcal{I}$, we may assume $H$ is very sparse compared to the strong neighborhood of $\chi$ and those of the $\chi_i$. Then an easy calculation shows $\chi^{Hg} = [\chi_1^{Hg}, \chi_2^{Hg}]$. Namely the strong neighborhoods of different translates of $\chi$ and $\chi_i$ do not intersect, thus these translates commute, thus for finite $F \Subset H$ we have
\begin{align*}
\chi^{Fg} &= \prod_{h \in F} (\chi_1^{-1})^{hg} \circ (\chi_2^{-1})^{hg} \circ \chi_1^{hg} \circ \chi_2^{hg} \\
&= \prod_{h \in F} (\chi_1^{-1})^{hg} \circ \prod_{h \in F} (\chi_2^{-1})^{hg} \circ \prod_{h \in F} \chi_1^{hg} \circ \prod_{h \in F} \chi_2^{hg} \\
&= [\chi_1^{Fg}, \chi_2^{Fg}] \end{align*}
and the claim follows by taking the (uniformly converging) limits and applying Lemma~\ref{lem:UniformComposition}.

Next, we show $[\gnets_{\mathcal{I}}, \gnets_{\mathcal{I}}] \subset \egnets_{\mathcal{I}}$. The idea is to take a generic commutator, and use Lemma~\ref{lem:Decomposition} to decompose the commutatorands into very sparse gate lattices. (This is where we use the fact our net $\mathcal{I}$ contains arbitrarily sparse finite-index subgroups.) Standard commutator formulas allow us to express the original commutator as a composition of commutators of sparse gate lattices. A simple calculation then shows that a commutator of very sparse gate lattices is the same as applying a commutator of gates on a sparse lattice, the gist being that interactions between the commutatorands ``happen only in the intended locations''. We conclude by recalling that a commutator of gates is an even gate.

In detail, consider any commutator $[\chi_1^{Hh}, \chi_2^{Kk}]$, for some $h,k \in G$. Repeatedly using the standard formulas $[a, cb] = [a, b][a, c]^b$ and $[ac, b] = [a, b]^c [c, b]$, and Lemma~\ref{lem:Decomposition}, we can write $[\chi_1^{Hh}, \chi_2^{Kk}]$ as a composition of elements of the form $[\chi_1^{Lh'}, \chi_2^{Lk'}]^f$ where $f \in \gnets_{\mathcal{I}}$ and $L \in \mathcal{I}$ is a normal finite-index subgroup of $G$ which is very sparse compared to the radii of $\chi_1, \chi_2$. Specifically we want $L \cap N^{-1}NN^{-1}N = \{e_G\}$, where $N$ is a common strong neighborhood for $\chi_1, \chi_2$.


It suffices now to show that $[\chi_1^{Lh'}, \chi_2^{Lk'}] = [\chi_1^{h'L}, \chi_2^{k'L}] \in \egnets_{\mathcal{I}}$, since the group $\egnets_{\mathcal{I}}$ is normal. First, up to possibly changing the representatives $h', k'$, we may assume that either $Nh' \cap Nk' \neq \emptyset$, or $Nh'\ell \cap Nk'\ell' = \emptyset$ for any $\ell, \ell' \in L$. In the latter case, clearly $[\chi_1^{h'L}, \chi_2^{k'L}] = \id \in \egnets_{\mathcal{I}}$, since $\chi_1^{h'L}$ applies rewrites in the domain $Nh'L$, and $\chi_2^{k'L}$ applies rewrites in $Nk'L$.

Consider then the former case, $Nh' \cap Nk' \neq \emptyset$,
thus $h'(k')^{-1} \in N^{-1} N$. In this case, we must have $Nh' \cap Nk'' = \emptyset$ for all $k'' = k'\ell \in k'L \setminus \{k'\}$, namely otherwise $h'(k'')^{-1} \in N^{-1} N$, so $k''(k')^{-1} \in N^{-1} N N^{-1} N$. But we have $k''(k')^{-1} = k' \ell (k')^{-1} \in L$ (because $L$ is normal), contradicting the sparseness assumption on $L$.


By similar logic, since $N h' \ell \cap N k' \ell \neq \emptyset$ we have $N h' \ell \cap N k'' \ell \neq \emptyset$ for $k'' \in k'L \setminus \{k'\}$. It is now clear that $[\chi_1^{h'L}, \chi_2^{k'L}] = [\chi_1^{h'}, \chi_2^{k'}]^L$, since the only interaction between applications of $\chi_1$ and $\chi_2$ happen in the intersections $N h' \ell \cap N k' \ell$. Now finally note that $[\chi_1^{h'}, \chi_2^{k'}] \in \egates_{\mathcal{I}}$ by Lemma~\ref{lem:CommutatorCharacterization2}, so by definition $[\chi_1^{h'}, \chi_2^{k'}]^L \in \egnets_{\mathcal{I}}$, concluding the proof.
\end{proof}

We show that the difference between $\gnets_{\mathcal{I}}$ and $\egnets_{\mathcal{I}}$ is only about parity issues, as the names might suggest. Recall that a \emph{Boolean group}, also known as an \emph{elementary abelian $2$-group}, is a group where every element is an involution.

\begin{proposition}
\label{prop:AbelianizationTwo}
The abelianization of $\gnets_{\mathcal{I}}$ is a Boolean group.
\end{proposition}

\begin{proof}
It suffices to show that $\gnets$ admits a generating set of elements whose squares are in $[\gnets, \gnets]$. Consider any gate $\chi$. We have $\chi^2 \in \egnets$ because for any strong neighborhood $N$ of $\chi$, $\chi^2$ performs an even permutation in any $G \setminus N$-context (because the abelianization of a nontrivial finite symmetric group is $\Z_2$). If $Hg$ is sparse enough, then $(\chi^{Hg})^2 = (\chi^2)^{Hg}$, and thus $(\chi^{Hg})^2 \in \egnets$. Since $\egnets \subset [\gnets, \gnets]$ by Lemma~\ref{lem:egnetsIsCommu}, this concludes the proof.
\end{proof}

In the remainder of this section we find conditions under which $\egnets_{\mathcal{I}} = \gnets_{\mathcal{I}}$.

We first show that a full shift over an even alphabet over any residually finite group $G$ has this property. We in fact state some more general dynamical properties that imply it. First say a subshift has \emph{even fillings} if for some $F \Subset G$, we have that $|\follow(x, F)|$ is even for all $x \in X|G \setminus F$. It is clear that if this happens for some $F$, it happens for all larger $F$ and all right translates of $F$.

\begin{lemma}
\label{lem:EvenPerfect}
If $X$ is an MFP SFT with even fillings, then $\egnets_{\mathcal{I}} = [\gnets_{\mathcal{I}}, \gnets_{\mathcal{I}}] = \gnets_{\mathcal{I}}$ for any lattice net $\mathcal{I}$.
\end{lemma}

\begin{proof}
If $\chi$ has strong neighborhood $N$, let $R$ be a window for $X$, let $F$ have even fillings, and replace the strong neighborhood of $\chi$ by $N \sqcup Fg$ such that $N \cap RFg = \emptyset$, acting trivially on the contents of $Fg$. Now consider any context $x \in X|G \setminus (N \cup Fg)$. Since applying $\chi$ does not affect the contents of $RFg$, the possible contents of $Fg$ only depends on $x$ before and after the application.\footnote{Since $\chi$ is well-defined with strong neighborhood $N$, it is not even necessary to assume $N \cap RFg = \emptyset$, only that $N \cap Fg = \emptyset$; but the extra leeway does not hurt.} Thus for each even cycle that $\chi$ performs, we perform an even number of independent copies of it when $\chi$ is seen as a permutation of $\follow(x, N \sqcup Fg)$, in particular this is an even permutation. It is clear that $\chi$ with the new neighborhood still commutes with its translates (since it is the same homeomorphism).
\end{proof}

A common way to get even fillings is that we literally have freely changeable bits visible on each configuration. More precisely, we say a subshift has \emph{syndetic free bits} if it is conjugate to a subshift $X$ over the disjoint union alphabet $A \cup (B \times \{0,1\})$ such that elements of $B \times \{0,1\}$ appear in every configuration, and if $x \in X$ and $x_g = (b, c) \in B \times \{0,1\}$, then also $y \in X$, where $y_h = x_h$ for $h \neq g$ and $y_g = (b, 1-c)$.

\begin{lemma}
\label{lem:FreeBitsPerfect}
If $X$ has syndetic free bits, then is has even fillings.
\end{lemma}

\begin{proof}
A free involution on the fillings of a hole $F$ is obtained by flipping the first free bit under any ordering of $F$ (and a free involution implies even cardinality). The hole $F$ simply needs to be large enough so there is always a free bit. Indeed, since free bits appear in all configurations of $X$, by compactness they appear in any large enough finite pattern of any configuration, thus such $F$ exists.
\end{proof}

Next, we cover all full shifts, under a condition on the group. We say a group $G$ \emph{has halvable finite-index subgroups} if all of its finite index subgroups have a finite index subgroup with an even index, or in a formula
\[ \forall H \leq G: [G : H] < \infty \implies \exists K \leq H: [H : K] < \infty \wedge 2|[H : K]. \]


\begin{lemma}
\label{lem:HalvablePerfect}
If $G$ has halvable finite-index subgroups and $X = \Sigma^G$, then $\egnets = [\gnets, \gnets] = \gnets$.
\end{lemma}

\begin{proof}
It suffices to write any $\chi^{Hg}$ as a commutator, where $H$ is a sufficiently sparse normal finite index subgroup. Let $K \leq H$ have even index, take right coset representatives $k_1, \ldots, k_{2n} \in H$. By normality of $K$ we have $\chi^{Kk_i} = \chi^{k_iK}$ and we can see $\chi^{Hg}$ as a composition of maps $(\chi^{k_{2i+1}} \circ \chi^{k_{2i+2}})^{Kg}$.


Clearly $\chi^{k_{2i+1}}$ and $\chi^{k_{2i+2}}$ have the same parity in any annular context because the subshift's coordinates are interchangeable (we are on a full shift). Thus the parity of their composition is even in every context. By MFP and Lemma~\ref{lem:CommutatorCharacterization} we can write their product as a commutator.
\end{proof}

We give two examples of (classes of) groups with halvable finite-index subgroups.

\begin{lemma}
\label{lem:HalvableExamples}
The following groups have halvable finite-index subgroups:
\begin{itemize}
\item all finitely-generated infinite finite-dimensional matrix groups over commutative rings;
\item all infinite residually finite $2$-groups.
\end{itemize}
\end{lemma}

\begin{proof}
The first item is a special case of a more general result in \cite{We87}. For the latter, it is clearly enough to prove that every infinite residually finite $2$-group has a subgroup of even finite index. Take any proper normal finite-index subgroup $N \triangleleft G$ and consider the group $G/N$. The order of any nontrivial element $gN$ is a power of $2$, so by Lagrange's theorem the order $[G : N]$ of $G/N$ is even.
\end{proof}

For a similar reason as the second item in the previous lemma, a $p$-group with $p \neq 2$ will never have halvable finite-index subgroups.

Of course one could generalize Lemma~\ref{lem:HalvablePerfect} to $\gnets_{\mathcal{I}}$, namely simply take it as a property of a lattice net $\mathcal{I}$ that any $H \in \mathcal{I}$ has an even index subgroup $K \in \mathcal{I}$. Note that even $\Z$ does not have this property for all lattice nets $\mathcal{I}$.

In Lemma~\ref{lem:HalvablePerfect}, we could also replace the condition of being a full shift with a weaker statement, namely that that number of connecting patterns between an annular pattern and another pattern positioned in two different fixed areas inside it have the same parity (with suitable quantifiers). We omit a precise statement, but in the case of $\Z$ we can use this idea to generalize the statement $\egnets = \gnets$ to all mixing SFTs.

\begin{lemma}
\label{lem:ZPerfect}
If $G = \Z$ and $X$ is a mixing SFT, then $\egnets = [\gnets, \gnets] = \gnets$.
\end{lemma}

\begin{proof}
It clearly suffices to show that, if $\chi$ is an arbitrary permutation with strong neighborhood $[0,n-1]$, then $\chi^{n\Z}$ is a composition of even gate lattices. We may assume $X$ is represented as an edge shift, so elements of $\Z$ carry edges of a finite directed graph and there are vertices between them at elements of $\Z+\frac12$; let $M$ be the matrix with $M_{a,b}$ the number of edges from vertex $a$ to vertex $b$. Note that the permutation $\pi$ that $\chi$ performs in $[0,n-1]$ fixes the vertices at $-\frac12$ and $n-\frac12$. Write $p_{a,b}$ for the parity of the restriction of $\pi$ to the context where the vertices at $(-\frac12, n-\frac12)$ are respectively $(a,b)$.

Next let $\hat M$ be the matrix obtained by taking entries of $M$ modulo $2$, and observe that $\hat M^n = \widehat{M^n}$. Observe that powers of $\hat M^n$ eventually get into a cycle, meaning $\hat M^{mn} = \hat M^{(m+p)n}$ for some $m \geq 0, p \geq 1$. Now consider the commuting product $\chi \circ \chi^{\sigma_{pn}}$. We claim that if we use the strong neighborhood $[-mn, mn+pn-1]$, this is an even permutation.

To see this, let $a,b \in V$ where $V$ is the set of vertices, and consider the permutation $\chi$ with vertices $(a,b)$ at $(-mn-\frac12, mn+pn-\frac12)$. For $c,d \in V$, for every choice of path from $a$ to $c$ and from $d$ to $b$, counting modulo $2$, $\chi$ has $p_{c,d}$ cycles of even length. If the number of paths from $a$ to $c$ is even or the number of paths from $d$ to $b$ is even, these cancel out, so the parity is just the parity of then number of triples $(a,c,d,b)$ such that $\hat M^{mn}_{a,d} = p_{c,d} = M^{(m+p)n}_{d,b} = 1$. The same calculation holds for $\chi^{\sigma_{pn}}$, so their composition performs an even permutation in the context $(a,b)$.

We conclude that $\chi \circ \chi^{\sigma_{pn}}$ with a suitable choice of strong neighborhood is even, and clearly the choice of strong neighborhood does not affect the commutation of the product $(\chi \circ \chi^{\sigma_{pn}})^{2np\Z} = \chi^{np\Z}$.

We can apply $\chi$ in a similar paired-up way on other cosets of $pn\Z$. Specifically we observe that
\[ (\chi \circ \chi^{\sigma_{pn}})^{\sigma_{in}} = \chi^{\sigma_{in}} \circ \chi^{\sigma_{(p + i)n}} \]
is also even, and define $f_i = (\chi^{\sigma_{in}} \circ \chi^{\sigma_{(p + i)n}})^{2np\Z} = \chi^{(n + i)p\Z}$.
Now $\chi^{n\Z} = \prod_{i = 0}^n f_i$. We get $\gnets \subset \egnets$. Since $\egnets \subset [\gnets, \gnets] \subset \gnets$, we conclude that the three groups are equal.
\end{proof}


\subsection{Inertness}
\label{sec:Inertness}

In the case $G = \Z$, recall from the introduction the notion of inertness of an automorphism of a mixing SFT, meaning an automorphism that acts trivially on Krieger's dimension group. 
We will mostly need a result of Wagoner \cite{Wa90a}, namely Lemma~\ref{lem:Wagoner} below. To state it, we need a few definitions.

Recall again that an edge shift is the set of paths $p : \Z \to E$ (with matching endpoints for successive edges) where $(V, E)$ is a directed (multi-)graph (with loops). A \emph{simple graph symmetry} is an automorphism of an edge shift which is defined (as a $1$-block code) by a bijection $\pi : E \to E$ that preserves the tails and heads of all vertices. If $X$ is an SFT, an automorphism $f \in \Aut(X, \sigma)$ is \emph{simple} if there exists a topological conjugacy between $(X, \sigma)$ and an edge shift, by which $f$ is conjugated to a simple graph symmetry. This definition is due to Nasu \cite{Na88}. The relevant result of Wagoner is the following:

\begin{lemma}
\label{lem:Wagoner}
If $f$ is an inert automorphism of a one-dimensional mixing SFT $(X, \sigma)$, then there exists $m \in \Z_+$ such that for all $n \geq m$, $f$ can be written as a product of simple automorphisms of $\Aut(X, \sigma^n)$.
\end{lemma}

\begin{lemma}
\label{lem:1DCase}
Let $X$ be a one-dimensional mixing SFT. Then the following are equivalent:
\begin{enumerate}
\item $f$ is an inert automorphism of $(X, \sigma^n)$ for some $n \in \Z_+$;
\item $f$ is a product of simple automorphisms of $(X, \sigma^n)$ for some $n \in \Z_+$;
\item $f$ is a product of gate lattices on $X$;
\item $f$ is a product of even gate lattices on $X$.
\end{enumerate}
\end{lemma}

\begin{proof}
(1) $\implies$ (2): If $f$ is inert for some $(X, \sigma^n)$, then because $(X, \sigma^n)$ is itself topologically conjugate to a mixing SFT, by Wagoner's result it is a product of simple automorphisms of higher powers of $\sigma^n$.

(2) $\implies$ (1): This is obvious from the definition of the dimension group through rays and beams \cite[§7.5]{LiMa95}. Namely, if we represent $X$ as the edge shift where the automorphism is a simple graph symmetry, the action on rays is clearly the identity map, so simple automorphism are inert. Conjugating an inert automorphism back to $X$ preserves inertness. We conclude by recalling that the inert automorphism form a group, as they are the kernel of a homomorphism.

(2) $\implies$ (3): If $f$ is a simple automorphism of $\sigma^n$, then the edge permutations are obviously commuting gates, and thus $f$ is directly a gate lattice on the subgroup $n\Z$ (note that looking through a topological conjugacy of course does not change the set of gates, nor affect the commutation of their translates).

(3) $\implies$ (2): We simply need to show that if $\chi^{n\Z}$ is a gate lattice, then it can be written as a product of simple automorphisms. This is straightforward from Lemma~\ref{lem:Decomposition}, namely we can write $n\Z$ as a union of sparser translated subgroups $jn + mn\Z$, and it is easy to see that for large enough $m$, $\chi^{jn+mn\Z}$ is a simple automorphism of $(X, \sigma^{mn\Z})$: One can take the vertices to represent words of length $r$ around the positions $j + \lfloor mn/2 \rfloor + kmn$, where $[0,r-1]$ is a window for $X$. For large $m$, these words are not modified by $\chi^{j+mn\Z}$, and indeed if the strong neighborhood of $\chi$ is contained in $[-mn/2+r, mn/2-r]$ then $\chi$ can be seen as a permutation of the edges.

The equivalence of (3) and (4) is Lemma~\ref{lem:ZPerfect}.
\end{proof}

If $G = \Z$ and $X$ is a mixing SFT, the \emph{stabilized inert automorphism group} is the smallest group containing the inert automorphisms of the mixing SFTs $(X, \sigma^n)$, for all $n \geq 1$. The following is immediate from Lemma~\ref{lem:1DCase}.

\begin{proposition}
\label{prop:InertIsEgnets}
The stabilized inert automorphism group of a topologically mixing $\Z$-SFT is equal to its group $\egnets$.
\end{proposition}

\section{Proof of main results}

We now prove the main result, Theorem~\ref{thm:MainProof} (which is our main result from the introduction, generalized to allow general lattice nets). 
The structure of the proof of simplicity and monolithicity of $\egnets$ is the following:
\begin{enumerate}
\item We commutator an arbitrary element of $\egnets$ (or more generally a homeomorphism with ntinuous inverse) with a suitably chosen element of $\egates$ to obtain a non-trivial map (first two lemmas).
\item We observe that applying the same on a sufficiently sparse lattice gives us a nontrivial gate lattice (Lemma~\ref{lem:KinCommutator}).
\item We observe that once we have one gate on sparse enough lattices, by EFP we have all of them on sparse enough lattices (done in the proof of Lemma~\ref{lem:GeneratingEgnets}).
\item Once we can apply arbitrary gates on sparse enough lattices, we can apply them on any lattice by refinement (done in the proof of Lemma~\ref{lem:GeneratingEgnets}).
\end{enumerate}



\begin{lemma}
\label{lem:CommuIsGate}
Let $f :X \to X$ be a homeomorphism with ntinuous inverse on a subshift $X$, and $\chi$ be a gate. Then $[f, \chi]$ is a gate.
\end{lemma}

\begin{proof}
To see that $[f, \chi] = f^{-1} \circ \chi^{-1} \circ f \circ \chi$ is a gate, it suffices to show that $(\chi^{-1})^f$ is a gate, and since $\chi$ is arbitrary it suffices to show that $\chi^f$ is. This is Lemma~\ref{lem:GatesNormal}.
\end{proof}

\begin{lemma}
\label{lem:CommuIsNontrivial}
Let $f : X \to X$ be a nontrivial homeomorphism with ntinuous inverse. If $X$ is an SFT with MFP then $[f, \chi]$ is a nontrivial gate, for some gate $\chi \in \egates$.
\end{lemma}

\begin{proof}
Let $R$ be a window for $X$. Let $x\in X$ be such that $f(x)_g \neq x_g$ for some $g \in G$, and pick a large enough $r$ so that any pattern $P$ on the annulus $A_{r,R}$ has at least $3$ fillings for $B_r$. By gates in $\egates$ we can realize any even permutation of $\mathcal{F}(P, RB_r)$ for any $P \in X|A_{r,R}$. 
Now pick $\chi$ any nontrivial $3$-rotation of patterns in $\mathcal{F}(x|A_{r,R}, RB_r)$ which has $x$ in its support, but not $f(x)$. If $x|A_{r,R} \neq f(x)|A_{r,R}$, this is trivial, otherwise it follows from $|\mathcal{F}(x|A_{r,R}, B_r)| \geq 3$, since of the at least three different fillings, at most one can agree with $f(x)$. Now $f$ cannot commute with $\chi$, as it does not preserve its support, so $[f, \chi]$ is nontrivial. By the previous lemma, it is a gate.
\end{proof}

\begin{lemma}
\label{lem:KinCommutator}
Let $f \in \Aut(X,H)$ for some $H \leq G$ of finite index, and $\chi$ a gate. If $K \leq H$ is sparse enough, then we have $[f, \chi]^K = [f, \chi^K]$ (and these expressions are well-defined).
\end{lemma}

\begin{proof}
Since $f$ is $H$-invariant, it is bintinuous. Since $[f, \chi]$ is a gate, and $f$ is also $K$-invariant, we have $f^k = f, (f^{-1})^k = f^{-1}$, so
\[ [f, \chi]^K = \prod_{k \in K} [f, \chi]^k = \prod_{k \in K} f^{-1} \circ (\chi^{-1})^k \circ f \circ \chi^k \]
is well-defined for sparse enough $K$.

Since $f$ is $H$-invariant and $K \leq H$, for sparse enough $K$, $\chi^k$ commutes with $((\chi^{-1})^{k'})^f$ whenever $k \neq k'$. Namely, we have
\[ ((\chi^{-1})^{k'})^f = f^{-1} k'^{-1} \chi^{-1} k' f = k'^{-1} f^{-1} \chi^{-1} f k' = ((\chi^{-1})^f)^{k'}, \]
and it follows from Lemma~\ref{lem:GatesNormal} that $(\chi^{-1})^f$ admits a strong neighborhood $N'$. If $N$ is the strong neighborhood of $\chi$, it suffices that $Nk \cap N'k' = \emptyset$ for $k \neq k'$, which is just the condition $k'k^{-1} \notin N'^{-1}N$ and holds for sparse enough $K$.

Thus, finite subproducts of $[f, \chi]^K$ can be rearranged to approximations of $[f, \chi^K]$, in the sense that if $F \subset K$ is finite then
\begin{align*}
\prod_{k \in F} f^{-1} \circ (\chi^{-1})^k \circ f \circ \chi^k &= \prod_{k \in F} ((\chi^{-1})^k)^f \circ \chi^k \\
&= \prod_{k \in F} ((\chi^{-1})^k)^f \circ \prod_{k \in F} \chi^k \\
&= [f, \chi^F],
\end{align*}
thus the two infinite products are also the same.
\end{proof}

\begin{lemma}
\label{lem:AnSimplePlus}
For any $n \geq 5$ and $f \in S_n \setminus \{e_{S_n}\}$, the smallest subgroup of $S_n$ containing $f$ and invariant under conjugation by $A_n$ contains $A_n$.
\end{lemma}

\begin{proof}
If $f \in S_n \setminus A_n$ is not the identity permutation, observe that $f$ cannot commute with every $3$-cycle, as it does not preserve the support of every $3$-cycle, and thus we find a nontrivial commutator between $f$ and $f^{\pi}$ for a $3$-cycle $\pi$. Thus we have generated a nontrivial element of $A_n$. If $f \in A_n$, the claim is clear since $A_n$ is simple.
\end{proof}


\begin{lemma}
\label{lem:GeneratingEgnets}
Let $X$ be an EFP SFT on a residually finite countably infinite group $G$, and let $f \in \Aut(X, H)$ be nontrivial where $H \leq G$ is of finite index. Then
\[ \langle f^{\egnets_{\mathcal{I}}} \rangle = \langle f, \egnets_{\mathcal{I}} \rangle \]
for any lattice net $\mathcal{I}$ containing $H$.
\end{lemma}

In words, the conclusion is that the smallest subgroup of $\Homeo(X)$ which contains all conjugates of $f$ by even gate lattices actually contains all even gate lattices.

\begin{proof}
Take an arbitrary nontrivial element $f \in \Aut(X, H)$. By Lemma~\ref{lem:AutBint}, $f$ is bintinuous. Since $f$ it is nontrivial, by Lemma~\ref{lem:CommuIsNontrivial}, $[f, \chi]$ is a nontrivial gate for some $\chi \in \egates$. Writing $\chi_0 = [f, \chi]$, we have by Lemma~\ref{lem:KinCommutator} that $\chi_0^K = [f, \chi^K]$ for any sparse enough $K \in \mathcal{I}$. Since 
$[f, \chi^K] = f^{-1} f^{\chi^K}$, we have that
$\langle f^{\egnets_{\mathcal{I}}} \rangle$ contains these gate lattices $\chi_0^K$.

We next show that from these, one can generate all even gate lattices $\chi^K$ where $K \leq H$, $K \in \mathcal{I}$. Let $R$ be a window for $X$, let $N$ be a strong neighborhood such that $\chi_0$ performs a nontrivial permutation of $X|N$, and large enough so that the cardinality of the latter set is at least $5$. If we pick a large enough $r$, then every $A_{r,R}$-context $P \in X|A_{r,R}$ allows an extension to any pattern in $N$ by EFP. Thus with the strong neighborhood $RB_r$, $\chi_0$ performs a permutation that fixes the $A_{r,R}$-context and acts nontrivially on the $B_r$-continuation for any context.

In a fixed such context $P \in X|A_{r,R}$ we can represent any even permutation of its extension patterns to $RB_r$ (i.e.\ the set $\follow(P, RB_r)$) as a composition of conjugates of $\chi_0$ by commutators with even permutations that fix the context $P$, by Lemma~\ref{lem:AnSimplePlus}. Composing these representations over all contexts, we can represent any even gate as a commutator expression involving only $\egates$-conjugates of $\chi_0$. Doing this simultaneously on all positions of a sparse enough finite-index subgroup $K \in \mathcal{I}$, we obtain for any even gate $\chi$ all elements of the form $\chi^K$ for all sparse enough $K \in \mathcal{I}$.

Consider now $K \leq H$, $K \in \mathcal{I}$ and suppose $\chi$ is an even gate and $K$ is sparse enough so that $\chi^K$ can be built with the above construction. Observe that since $f \in \Aut(X, H)$ we have $f^h = f$ for all $h \in H$, and thus if we conjugate the expression for $\chi^K$ (which is a composition of conjugates of $f$ by elements of $\egnets_{\mathcal{I}}$), the $h$-translation only affects the $\egnets_{\mathcal{I}}$-elements, so actually we obtain $\chi^{Kh}$ for cosets of $K$ in $H$.

To say the same in formulas, if
\[ \chi^K = \prod_i f^{\chi_i^{K_i}} \]
(note that this is a finite ordered product; actually $K_i = K$ for all $i$ in our construction, but this is immaterial) then using Lemma~\ref{lem:LatticeBase}
we have
\begin{align*}
\prod_i f^{\chi_i^{K_ih}} &= \prod_i (\chi_i^{-1})^{K_ih} \circ f \circ \chi_i^{K_ih} \\
&= \prod_i ((\chi_i^{-1})^{K_i})^h \circ f^h \circ (\chi_i^{K_i})^h \\
&= \prod_i (f^{\chi_i^{K_i}})^h \\
&= (\prod_i f^{\chi_i^{K_i}})^h \\
&= (\chi^K)^h = \chi^{Kh}.
\end{align*}

Consider next an arbitrary $L \leq H$, $L \in \mathcal{I}$, where again $\chi$ is an even gate and $\chi^L$ commutes. Let $K \leq L, K \in \mathcal{I}$ be any sparse enough subgroup such that $\chi^K$ is generated by $f^{\egnets_{\mathcal{I}}}$. Then by the two previous paragraphs, $\chi^{Kh}$ can also be built, where $h$ runs over right coset representatives for $K$ in $L$. By Lemma~\ref{lem:Decomposition}, we can build $\chi^L$. (To recall the argument, since $\chi^K$ commutes, the composition of the $\chi^{Kh}$ over coset representatives is precisely $\chi^L$.) We conclude that $\langle f^{\egnets_{\mathcal{I}}} \rangle$ contains every gate lattice $\chi^{Lh}$ where $\chi$ is even, $L \leq H$ and $h \in H$ is arbitrary.

Next, consider an arbitrary commuting $\chi^L$ for $L \in \mathcal{I}$, where not necessarily $L \leq H$. By what we already showed, $\chi^K$ in $\langle f^{\egnets_{\mathcal{I}}} \rangle$ where $K \leq L \cap H, K \triangleleft G, K \in \mathcal{I}$ is arbitrary. Now conjugating the situation with some $g \in G$, we transform the original map $f$ into some $f^g$, which is still a nontrivial homeomorphism that now commutes with $H^g$, and applying the entire discussion to it (observing $K \leq H^g$ by $K \leq H$ and normality), conjugates of $f^g$ by elements of $\egnets_{\mathcal{I}}$ also generate the same gate lattice $\chi^K$.

For some $\chi_i, K_i$ we now have
\[ \chi^K = \prod_i (f^g)^{\chi_i^{K_i}}, \]
and we can continue with
\begin{align*}
\chi^K = \prod_i(f^g)^{\chi_i^{K_i}} &= \prod_i(\chi_i^{-1})^{K_i} \circ g^{-1} \circ f \circ g \circ \chi_i^{K_i} \\
&= \prod_i {g^{-1}} \circ ((\chi_i^{-1})^{K_i})^{g^{-1}} \circ f \circ (\chi_i^{K_i})^{g^{-1}} \circ g \\
&= (\prod_i (\chi_i^{-1})^{K_i g^{-1}} \circ f \circ \chi_i^{K_i g^{-1}})^g.
\end{align*}
Conjugating both sides by $g^{-1}$, we see that $\langle f^{\egnets_{\mathcal{I}}} \rangle$ contains the element $(\chi^K)^{g^{-1}} = \chi^{Kg^{-1}}$ for arbitrary $g \in G$. In particular by composing these with $g^{-1} = t h$ where $t$ ranges over right coset representatives of $K$ in $L$, we obtain precisely $\chi^{Lh}$.
\end{proof}

\begin{lemma}
\label{lem:Simple}
Let $X$ be an EFP SFT on a residually finite countably infinite group $G$, and let $\mathcal{I}$ be any lattice net. Then the group $\egnets_{\mathcal{I}}$ is simple.
\end{lemma}

\begin{proof}
Otherwise, let $f \in \egnets_{\mathcal{I}}$ be arbitrary, i.e.\ $f$ is a product of some elements $(\chi_i)^{H_i g_i}$ where the $\chi_i$ are even gates. Picking any normal subgroup $H \in \mathcal{I}$ of $G$ contained in the intersection of the $H_i$, Lemma~\ref{lem:Decomposition} shows that $f$ is a product of gates of the form $(\chi_i)^{H g_i}$ with $H$ normal. By Lemma~\ref{lem:NormalAuto}, $f$ is an automorphism for the $H$-action. By Lemma~\ref{lem:GeneratingEgnets},
\[ \langle f^{\egnets_{\mathcal{I}}} \rangle = \langle f, \egnets_{\mathcal{I}} \rangle \]
so the smallest normal subgroup of $\egnets_{\mathcal{I}}$ containing $f$ is all of $\egnets_{\mathcal{I}}$. This concludes the proof of simplicity.
\end{proof}

\begin{lemma}
\label{lem:StabUnderAut}
Let $X$ be a subshift on a residually finite countably infinite group $G$, and ${\mathcal{I}}$ a lattice net. Then $\gnets_{\mathcal{I}}$ is a normal subgroup of the stabilized automorphism group $\SAut(X, \mathcal{I})$. Furthermore, if in addition $X$ is an MFP SFT, then $\egnets_{\mathcal{I}}$ is also a normal subgroup of $\SAut(X, \mathcal{I})$.
\end{lemma}

\begin{proof}
It suffices to show $(\chi^{Kg})^f \in \gnets_{\mathcal{I}}$ whenever $f \in \Aut(X,H)$ for finite-index $H \in \mathcal{I}$ and $K \in \mathcal{I}$ is a sparse enough finite-index subgroup of $H$, which is normal in $G$. 
We calculate
\begin{align*}
(\chi^{Kg})^f &= (\prod_{k \in K} \chi^{kg})^f = (\prod_{k \in K} \chi^{gk})^f = \prod_{k \in K} (\chi^{gk})^f \\
&= \prod_{k \in K} f^{-1}k^{-1}g^{-1}\chi gk f \\
&= \prod_{k \in K} k^{-1} f^{-1} g^{-1}\chi g f k \\
&= \prod_{k \in K} k^{-1} \chi^{f'} k, \mbox{ where } f' = gf \\
&= (\chi^{f'})^K
\end{align*}

Here the second equality follows from the normality of $K$. The third equality follows from bintinuity of $f$, Lemma~\ref{lem:UniformComposition} and commutation of $(\chi^{kg})^f$ for different $k \in K$; and the functions in the composition on the right side of the third equality commute, because the functions in the original composition commute. The fifth equality holds because $K \leq H$ and $f \in \Aut(X, H)$ imply $f \in \Aut(X, K)$.


By Lemma~\ref{lem:GatesNormal}, conjugating gates by $f \in \Aut(X,H)$ produces gates, so $(\chi^{f'})^K$ is indeed a gate lattice. Under the additional assumption, if $\chi$ is even then so is $\chi^{f'}$, by Lemma~\ref{lem:GatesNormal} and Lemma~\ref{lem:CommutatorCharacterization}, and thus $\egnets$ is also normal.
\end{proof}

Again recall that the \emph{monolith} of a group, if it exists, is its unique maximal non-trivial normal subgroup.

\begin{theorem}
\label{thm:MainProof}
For any residually finite countably infinite group $G$, any EFP SFT $X \subset \Sigma^G$, and any lattice net $\mathcal{I}$,
\begin{itemize}
\item $\egnets(X)_{\mathcal{I}}$ is simple,
\item $\egnets(X)_{\mathcal{I}}$ equals the commutator subgroup of $\gnets(X)_{\mathcal{I}}$,
\item $\egnets(X)_{\mathcal{I}}$ and $\gnets(X)_{\mathcal{I}}$ are normal in $\SAut(X)_{\mathcal{I}}$, and
\item $\egnets(X)_{\mathcal{I}}$ is the monolith of $\gnets(X)_{\mathcal{I}}$ and $\SAut(X)_{\mathcal{I}}$.
\end{itemize}
\end{theorem}

\begin{proof}
The first item is Lemma~\ref{lem:Simple}. The second item is Lemma~\ref{lem:egnetsIsCommu}. The third item is Lemma~\ref{lem:StabUnderAut}. For the fourth item, consider any non-trivial normal subgroup of $\SAut(X)_{\mathcal{I}}$ or $\gnets_{\mathcal{I}}$. Let $f$ be any element of this subgroup, so in particular $f \in \Aut(X, H)$ for some $H \in \mathcal{I}$. By Lemma~\ref{lem:GeneratingEgnets}, 
\[ \langle f^{\egnets_{\mathcal{I}}} \rangle = \langle f, \egnets_{\mathcal{I}} \rangle, \]
i.e.\ already $\egnets_{\mathcal{I}}$-conjugates of $f$ generate a group containing $\egnets_{\mathcal{I}}$. 
\end{proof}

\begin{remark}
For $G = \Z^d$, Theorem~\ref{thm:MainProof}, Proposition~\ref{prop:IsMoreGeneral} and Lemma~\ref{lem:Zdcofinality} show that the monolith of $\SAut(X, G)$ has a rich internal structure of infinite simple subgroups (the subgroups $\egnets_{\mathcal{I}}$).
\end{remark}

\section{Open problems}

\begin{question}
Can we define a notion of inert automorphism for subshifts on groups other than $\Z$ (through the introduction of a dimension group, or something in that spirit) so that $\egnets$ is recovered?
\end{question}

A related question is whether $\Aut(X) \cap \egnets(X)$ is the ``right'' notion of inertness in automorphism groups.

\begin{question}
Does $\gnets = [\gnets, \gnets]$ hold on all EFP SFTs and all residually finite infinite groups?
\end{question}

If equality does not always hold, one may ask if the rank can be infinite. Of course one could ask such questions also for general lattice nets.

\begin{question}
To what extent can we generalize the results of this paper beyond EFP SFTs?
\end{question}

Section~\ref{sec:LatticeNets} suggests the following technical questions.

\begin{question}
In Proposition~\ref{prop:IsMoreGeneral}, can we replace the far-from-cofinality assumption with natural dynamical assumptions on $X$?
\end{question}

\begin{question}
Is there an example of a pair of lattice nets $\mathcal{I}, \mathcal{J}$ on some residually finite group, such that $\mathcal{J}$ is neither cofinal in $\mathcal{I}$ nor far from cofinal in $\mathcal{I}$?
\end{question}

\section*{Acknowledgements}

We thank the anonymous referees for very detailed comments, which improved the paper throughout. In particular, we thank one referee for pushing us to prove Lemma~\ref{lem:egnetsIsCommu}, which clarified the message of the paper considerably. Proposition~\ref{prop:AbelianizationTwo} is due to the referee.

\bibliographystyle{plain}
\bibliography{../../../bib/bib}{}

\newpage
\appendix
\section{Gates as a topological full group}
\label{sec:GatesTFG}

We make a simple remark for readers interested in equivalence relations and groupoids, namely we explain how gates can be seen as a topological full group of a natural groupoid associated to the asymptotic relation. We simply wish to point out that gates are not really a ``new object'', but a special case of a standard construction.

We use the conventions from \cite{Ma16}, except our groupoids are denoted $H = (H^{(0)}, H^{(1)}, H^{(2)}, s, r)$ (\cite{Ma16} uses $G$ in place of $H$, but it is the name of our ambient group). We omit standard groupoid-related definitions; the reader can find them in \cite{Ma16} or any other reference.

One way of turning the asymptotic relation into an étale groupoid is to take $H^{(0)} = X$ and
\[ H^{(1)} = \{(x, y, F) \;|\; F \Subset G \wedge x, y \in X \wedge \diff(x, y) = F \}, \]
where $\diff(x, y) = \{g \in G \;|\; x_g \neq y_g\}$. We identify $H^{(0)}$ with the tuples $(x, x, \emptyset)$. We use the discrete topology on the finite sets $F$ (i.e.\ the subset of $H^{(1)}$ using a particular $F$ is open), and for each $F$ we use the topology induced from $X^2$ on the pair $(x, y)$. The source and target maps $s, r$ are defined by $s(x, y, F) = x$ and $r(x, y, F) = y$. The groupoid operation is $(x,y, F_1)(y, z, F_2) = (x, z, \diff(x, z))$. We call this the \emph{asymptoticity groupoid}.

If $s, r$ are injective on $U \subset H^{(1)}$, we call $U$ an \emph{$H$-set}, and $\tau_U : s(U) \to r(U)$ is the homeomorphism defined by $\tau_U(x) = r((s|_U)^{-1}(x))$. The following is Definition~4.1 from \cite{Ma16}.

\begin{definition}
Suppose $H$ is an essentially principal \'etale groupoid, and the unit space $H^{(0)}$ is Cantor. Then the set of $f : H^{(0)} \to H^{(0)}$ for which there exists a compact open $H$-set $U$ satisfying $f = \tau_U$ is called the \emph{topological full group} of $H$, denoted $\llbracket H \rrbracket$.
\end{definition}


In the following proposition, the standing assumption that $g$ is residually finite and countable is not needed.

\begin{proposition}
Let $X$ be an MFP SFT. Then its asymptoticity groupoid is a principal \'etale AF groupoid whose unit space is a Cantor space. A homeomorphism $\chi$ is a gate if and only if it is in the topological full group of the asymptoticity groupoid of $X$.
\end{proposition}

\begin{proof}
For any subshift, the groupoid is trivially principal (as in morphisms $(x, y, F)$ outside $H^{(0)}$ we require that $x, y$ differ in $F$). It easily seen to be second countable, and the unit space $H^{(0)} = X$ is a compact subset of Cantor space. In the case of an MFP SFT $X$, it is homeomorphic to Cantor space by Lemma~\ref{lem:EFPMFPCantor}.

The groupoid is étale. For this, we need to show that $s$ and $r$ are local homeomorphisms, i.e.\ every point has an open neighborhood $U$ which has open image, and $U$ is mapped homeomorphically to its image. We show this for $s$. Consider for any $(x, y, F) \in H^{(1)}$ the set $U = [x|F']_{F'} \times [y|F']_{F'} \times \{F\}$ where $F' \supset NF$, where $N$ is a window for $X$. Then the image is precisely the clopen set $[x|F']_{F'} \subset X$, in particular it is open.

To see this, we show that substituting the contents of $F$ in some $x' \in [x|F']_{F'}$ with $y|F$ does not lead to a forbidden pattern. If on the contrary replacing the contents in $F$ in $x' \in [x|F']_{F'}$ were to introduce a forbidden pattern $P$, say $sx|D = P$, then $Ds \cap F \neq \emptyset$ so for some $d \in D$ we have $ds \in F$, meaning $Ds \subset Dd^{-1}F \subset NF$, since by the definition of a window $N$ contains $DD^{-1}$ for the defining forbidden patterns. Since $F' = DF$ and $y|F'$ is a legal pattern in $X$, this is not possible. It is trivial that restricted to $U$ the map is a homeomorphism, so the groupoid is \'etale.
 
The groupoid is AF. Namely, for $F' \Subset G$ write $H_{F'} = H^{(1)} \cap X \times X \times \{F \;|\; F \subset F'\}$. This is easily seen to be a subgroupoid, and now letting $F_n$ be any increasing sequence of finite sets with union $G$, the sets $H_{F_n}$ can be taken as the compact approximations.

We now show the last claim. Suppose $\chi$ is a gate. Pick a strong neighborhood $R \Subset G$. If $\chi(P) = Q$, let $F_P = \diff(P, Q)$ for any two patterns $P, Q$, i.e.\ the set of coordinates where $P$ differs from its $\chi$-image, and define $U = \{(x, \chi(x), F_{x|R}) \;|\; x \in X\}$. Then $U$ is a compact open $H$-set and we have $\tau_U = \chi$ by a straightforward calculation. Conversely, if $\chi = \tau_U : X \to X$ is a homeomorphism, where $U \subset H$ is a compact open $H$-set, then $U \subset H_F$ for some finite set $F$ by the proof of the AF property, and then obviously $F$ is a weak neighborhood.
\end{proof}

\end{document}